\documentclass{amsart}

\usepackage{amsmath}
\usepackage{amsfonts}
\usepackage{amssymb}
\usepackage{amsthm}
\usepackage{array}
\usepackage{bbm}
\usepackage{chngpage}
\usepackage{comment}
\usepackage{float}
\usepackage[OT2,T1]{fontenc}
\usepackage{graphicx,caption}
\usepackage[export]{adjustbox}
\usepackage{longtable}
\usepackage{mathrsfs}
\usepackage{mathtools}
\usepackage{wrapfig}
\usepackage{cutwin}
\usepackage{shapepar}
\usepackage{tikz}
\usepackage[lowtilde]{url}
\usepackage{subcaption}
\usetikzlibrary{matrix,arrows}
\usepackage{tabularx}
\usepackage{multirow}
\usepackage[colorlinks, citecolor = blue]{hyperref}
\usepackage[capitalize, nameinlink]{cleveref}
\usepackage{enumitem}
\usepackage{diagbox}
\usepackage{etoolbox}
\usepackage{algpseudocode}
\usepackage{etoolbox}

\usepackage[margin=0.9in]{geometry}

\patchcmd{\subsection}{-.5em}{.5em}{}{}

\newcommand{\C}{\mathbb{C}}

\newcommand{\HH}{\mathbb{H}}

\newcommand{\R}{\mathbb{R}}
\newcommand{\Z}{\mathbb{Z}}

\newcommand{\cO}{\mathcal{O}}

\newcommand{\fa}{\mathfrak{a}}

\DeclareSymbolFont{cyrletters}{OT2}{wncyr}{m}{n}
\DeclareMathSymbol{\sha}{\mathalpha}{cyrletters}{"58}

\newcommand{\eps}{\varepsilon}

\newcommand{\ds}{\displaystyle}
\newlength{\strutheight}
\newcommand{\half}{\frac{1}{2}}
\newcommand{\thalf}{\tfrac{1}{2}}
\newcommand{\tth}{\text{th}}

\newtheorem{theorem}{Theorem}[section]
\newtheorem{lemma}[theorem]{Lemma}
\newtheorem{corollary}[theorem]{Corollary}

\newtheorem{proposition}[theorem]{Proposition}
\newtheorem{assumption}[theorem]{Assumption}

\theoremstyle{definition}
\newtheorem{definition}[theorem]{Definition}

\newtheorem{example}[theorem]{Example}
\AtBeginEnvironment{example}{%
  \pushQED{\qed}%
}
\AtEndEnvironment{example}{\popQED\endexample}

\author{Alex Cowan}
\address{Department of Mathematics, Harvard University, Cambridge, MA 02138 USA}
\email{cowan@math.harvard.edu}
\thanks{The author was supported by the Simons Foundation Collaboration Grant 550031.}

\title{A twisted additive divisor problem}
\date{\today}

\begin{document}
\maketitle
\begin{abstract}
  \noindent
  We give asymptotics for shifted convolutions of the form
  $$\sum_{n < X} \frac{\sigma_{2u}(n,\chi)\sigma_{2v}(n+k,\psi)}{n^{u+v}}$$
  for nonzero complex numbers $u,v$ and nontrivial Dirichlet characters $\chi,\psi$. We use the technique of \textit{automorphic regularization} to find the spectral decomposition of a combination of Eisenstein series which is not obviously square-integrable. The error term we obtain is in some cases smaller than what the method we use typically yields.
\end{abstract}
\tableofcontents
\section{Introduction}
\noindent
For any Dirichlet character $\chi$ let $\tau(\chi)$ denote the associated Gauss sum, and for any $s \in \C$, $n \in \Z$, let $\sigma_s(n,\chi)$ denote the sum of divisors function $\sigma_s(0,\chi) \coloneqq 0$ and
\begin{align*}
  \sigma_s(n,\chi) \coloneqq \sum_{\substack{d\mid n\\d > 0}}\chi(d)d^s.
\end{align*}
\begin{theorem}\label{maintheorem}
  For any
  \begin{itemize}
  \item positive integer $k$,
  \item rational prime $N$,
  \item even nontrivial Dirichlet characters $\chi$ and $\psi$ mod $N$ such that $\chi\psi$ is nontrivial,
  \item nonzero complex numbers $u$ and $v$ satisfying $|\Re(u)| + |\Re(v)| < \thalf$,
  \item $0 < \eps < \thalf - |\Re(u)| - |\Re(v)|$,
  \item positive integer $X$,
  \end{itemize}
  \begin{align*}
    \sum_{n=1}^X \frac{\sigma_{2u}(n,\chi)\sigma_{2v}(n-k,\psi)}{n^{u+v}}
    = \frac{L(1-2u,\chi)L(1-2v,\psi)}{L(2-2u-2v,\chi\psi)}\sigma_{-1+2u+2v}(k,\chi\psi)\frac{X^{1-u-v}}{1-u-v}&\\
    + \,\frac{\tau(\overline{\chi\psi})}{\tau(\overline{\chi})\tau(\overline{\psi})}\frac{L(1+2u,\overline{\chi})L(1+2v,\overline{\psi})}{L(2+2u+2v,\overline{\chi\psi})}\frac{\sigma_{1+2u+2v}(k,\chi\psi)}{k^{1+2u+2v}}\frac{X^{1+u+v}}{1+u+v}&\\
    + \,\cO\!\left(X^{1 + |\Re(u)| + |\Re(v)| - \frac{1 + 2|\Re(u)| + 2|\Re(v)|}{3 + |\Re(u+v)| + |\Re(u-v)|} + \eps}\right)&
  \end{align*}
  as $X \to \infty$ and all other quantities fixed.
\end{theorem}
\noindent
The quantity $$\frac{\sigma_{2u}(n,\chi)\sigma_{2v}(n-k,\psi)}{n^{u+v}}$$ appearing on the left hand side of \cref{maintheorem} is natural to consider in light of the functional equation
\begin{align*}
  \frac{\sigma_{2u}(n,\chi)}{n^u} = \chi(n)\frac{\sigma_{-2u}(n,\bar\chi)}{n^{-u}}
\end{align*}
for $(n,N) = 1$. In \cref{maintheorem}, one can replace the factor of $n^v$ in the denominator with $(n-k)^v$ without having to modify the error term as stated. If $(k,N) = 1$, then
\begin{align*}
  \frac{\sigma_{1+2u+2v}(k,\chi\psi)}{k^{1+2u+2v}} = \chi\psi(k) \,\sigma_{-1-2u-2v}(k,\overline{\chi\psi}),
\end{align*}
emphasizing the symmetry in the right hand side of \cref{maintheorem}.\\
\\
\cref{maintheorem} is an instance of the \textit{additive divisor problem}. The classical additive divisor problem is to estimate the sum
\begin{align}\label{additive_divisor_problem_equation}
  \sum_{n=1}^X \sigma_0(n) \sigma_0(n+k),
\end{align}
where $k$ is a positive integer, $\sigma_s(n) \coloneqq \sigma_s(n,\mathbbm{1})$, and $\mathbbm{1}$ denotes the Dirichlet character mod $1$. Additive divisor problems are instances of a more general type of problem, estimating \textit{shifted convolutions}, which we take to mean expressions of the form \eqref{additive_divisor_problem_equation} but with $\sigma_0$ replaced by any function of arithmetic interest. One very often considers \eqref{additive_divisor_problem_equation} with $\sigma_0$ replaced by the Fourier coefficients of some automorphic form, for example \cite{TV, motohashi, jutila, good, HH, HKLDW, HLN, NPR}.\\
\\
Shifted convolutions arise in a variety of contexts, some of which emphasize arithmetic statistics \cite{BD, HKLDW, BM}, and some of which emphasize analytic elements \cite{mueller, DFI, michel:subconvexity, young:4thmoment, young:QUE}, especially ``subconvexity bounds'' regarding the growth of $L$-functions on vertical lines. See \cite{michel} for an introductory survey.\\
\\
The most common sorts of estimates for quantities like \eqref{additive_divisor_problem_equation} consist of both a \textit{main term} and an \textit{error term}, which are, respectively, an asymptotic value of the sum \eqref{additive_divisor_problem_equation}, and a bound on the difference between the sum itself and this asymptotic value. Obtaining strong error terms is an interesting and difficult problem. For the classical additive divisor problem, Takhtadzhyan and Vinogradov \cite{TV} give an estimate with an error term of size $\text{main term}^{\frac{2}{3}+\eps}$, and this is typical when studying shifted convolutions via ``spectral methods'' like \cite{TV} and here. However, one often expects (and e.g.\ it is expected for the classical additive divisor problem \cite{TV, NPR}) that the ``true'' error term in these cases is of size $\text{main term}^{\frac{1}{2}+\eps}$, in the sense that the statements such as the main result of \cite{TV} are still true if one replaces $\tfrac{2}{3}$ by $\thalf$, but not by any smaller number. This barrier at $\tfrac{2}{3}$ can also be seen in \cite{jutila} and many other works.\\
\\
The main term of \cref{maintheorem} is of size $X^{1 + |\Re(u+v)|}$, provided that the coefficients are nonzero. The ratio of the exponent of the error term, minus $\eps$, to the exponent of the main term is
\begin{align}\label{error_term_ratio}
  &\frac{1 + |\Re(u)| + |\Re(v)| - \frac{1 + 2|\Re(u)| + 2|\Re(v)|}{3 + |\Re(u+v)| + |\Re(u-v)|}}{\ds{1 + |\Re(u+v)|}}.
\end{align}
For $\Re(u) = \Re(v) = \tfrac{1}{4}$, the ratio \eqref{error_term_ratio} is $\tfrac{13}{21}$, and for $\Re(u) = -\Re(v) = \tfrac{1}{4}$ the ratio \eqref{error_term_ratio} is $\tfrac{13}{14}$. All $u,v \in \C$ which satisfy the assumptions of \cref{maintheorem} yield values of \eqref{error_term_ratio} strictly between $\tfrac{13}{21}$ and $\tfrac{13}{14}$. The difference between the exponent of the main term and the exponent of the error term has similar behaviour, and ranges between $\tfrac{1}{14}$ and $\tfrac{8}{14}$.\\
\\
\cref{maintheorem} is, for some choices of $u, v$, a shifted convolution problem that breaks this ``$\frac{2}{3}$ barrier''. The techniques we use to bound this error term are largely standard for the field, and the relatively small size of the error term we obtain is a consequence of the locations of the poles of the cuspidal part of a spectral decomposition compared to the abscissa of convergence of the related $L$-series.\\
\\
The error term of \cref{maintheorem} omits the dependence on variables other than $X$. In the context of studying the related $L$-functions, this omission essentially corresponds to giving a bound of growth on vertical lines in the ``$t$-aspect''. We expect it'd be possible to give a dependence on the other variables appearing in \cref{maintheorem} using the same method as the one here and more technical analysis. Moreover, we expect that, if needed, one could write the error term as a sum of residues of explicit meromorphic functions plus an error of size $\cO(X^{-B})$ for any $B > 0$.\\
\\
The method by which we prove \cref{maintheorem} has been widely used to produce many similar results, particularly \cite{goldfeld}. See \cite{HH, HL20, HLN, nelson, MV} for fairly general treatments of shifted convolutions. Previous results have made various assumptions which we omit here, two of which we now highlight.\\
\\
First, previous results almost always only analyze the $L$-functions which appear on the line $\Re(s) = \thalf$. This would correspond to \cref{maintheorem} with $\Re(u) = \Re(v) = 0$; c.f.\ \cite{DFI} and \cite{HLN}.
\\
\\
Second, the overall strategy for estimating shifted convolutions is to use a spectral decomposition (e.g.\ \cref{spectral_decomposition}), and this requires certain automorphic forms appearing in the construction to be square-integrable. Most previous results satisfy this square-integrability condition by considering shifted convolutions involving at least one cuspidal automorphic form, for example \cite{HH, harcos, BH}. Other approaches are possible as well. For example, Goldfeld in \cite{goldfeld} considers essentially exactly the same shifted convolution we do, but considers completed Eisenstein series whose constant terms have factors of Dirichlet $L$-functions, and then ensures square-integrability by restricting to Eisenstein series with eigenvalues that cause these Dirichlet $L$-functions to vanish.\\
\\
Our approach to guaranteeing square-integrability is to subtract from our automorphic form of interest a specially chosen linear combination of Eisenstein series, a process we call \textit{automorphic regularization}. This is done more or less ad-hoc by \cite{TV, DFI, templier, PR, BD} and other places, and in those cases assumptions were made which allowed one to verify directly that the resulting difference would be square-integrable. In our setting, with level larger than $1$ and nontrivial nebentypus, it is not clear to us how to carry out similar direct inspection.\\
\\
Automorphic regularization was done used more systematically and more generally in \cite{MV} and \cite{wu}, though these can't be applied directly to our work here. Our approach is to take inspiration from these more systematic approaches and show, following \cite{HKLDW},  that we can guarantee a certain difference of automorphic forms is square-integrable, even though it is not clear to us how one would verify this directly. This process will also yield expressions for certain terms which appear in the spectral decomposition, allowing us to carry out the rest of the overall method. One can view this sort of automorphic regularization as a generalization of the Maass-Selberg relations, and to our knowledge was first treated in depth by Zagier in \cite{zagier}.

\section{Background}
\subsection{Bessel functions}
\noindent
The \textit{$K$-Bessel function} $K_u(y)$, often called the modified Bessel function of the second kind, is defined as \cite[8.432.7 with $z=1$]{GR}
\begin{align}\label{bessel_def}
  K_u(y) \coloneqq \half\int_0^\infty \exp\!\left(-y\frac{t+t^{-1}}{2}\right)t^u\frac{dt}{t}.
\end{align}
Via the transformation $t\mapsto\frac{1}{t}$, one sees that
$$K_u(y) = K_{-u}(y).$$
We will use the following Mellin transforms.
\begin{lemma}[{\cite[6.561.16]{GR}}]\label{bessel_mellin}
  If $\Re(a) > 0$ and $\Re(s) > |\Re(u)|$, then
  \begin{align*}
    &\int_0^\infty K_u(ay)y^s\frac{dy}{y} = 2^{s-2}a^{-s}\Gamma\!\left(\frac{s+u}{2}\right)\Gamma\!\left(\frac{s-u}{2}\right).
  \end{align*}
\end{lemma}
\begin{lemma}[{\cite[6.576.4]{GR}}]\label{bessel_product_mellin}
  If $\Re(a+b) > 0$ and $\Re(s) > |\Re(u)| + |\Re(v)|$, then
  \begin{align*}
    &\int_0^\infty K_u(ay)K_v(by)y^s \frac{dy}{y}\\
    &= \frac{2^{s-3}}{a^s\Gamma(s)}\left(\frac{b}{a}\right)^v\Gamma\!\left(\frac{s+u+v}{2}\right)\Gamma\!\left(\frac{s+u-v}{2}\right)\Gamma\!\left(\frac{s-u+v}{2}\right)\Gamma\!\left(\frac{s-u-v}{2}\right)F\!\left(\frac{s+u+v}{2},\frac{s-u+v}{2};s\,;1-\frac{b^2}{a^2}\right).
  \end{align*}
\end{lemma}
~
\subsection{Dirichlet $L$-functions}
\noindent
Given an even primitive Dirichlet character $\chi$, let
\begin{align}\label{completed_dirichlet_L_function_eqn}
  \Lambda(2s,\chi) \coloneqq \pi^{-s}N^s\Gamma(s)L(2s,\chi)
\end{align}
denote its completed Dirichlet $L$-function. It satisfies the functional equation
\begin{align*}
  \Lambda(s,\chi) = \frac{\tau(\chi)}{\sqrt{N}}\Lambda(1-s,\bar\chi),
\end{align*}
where $\tau(\chi)$ denotes the Gauss sum. When $\chi$ is the trivial character mod $1$, we will write $\Lambda(2s) \coloneqq \pi^{-s}\Gamma(s)\zeta(2s)$ to denote the completed Riemann zeta function.

\subsection{Hecke relations}\label{sec:hecke_relations}
\noindent
Let $f(z)$ be a Maass eigenform of level $N$, character $\chi$, eigenvalue $\tfrac{1}{4} + r^2$, and Hecke eigenvalues $a(n)$. Write the Fourier expansion of $f$ as
\begin{align*}
  \sum_{n\neq 0} 2 \rho\!\left(\frac{n}{|n|}\right) a(|n|) y^\half K_{ir}(2\pi|n|y)e(nx),
\end{align*}
Define
\begin{align*}
  &L(s,f) \coloneqq \sum_{n=1}^\infty \frac{a(n)}{n^s}\\
  \shortintertext{and, for any Dirichlet character $\psi$ modulo $N$,}
  &L(s,\psi\times f) \coloneqq \sum_{n=1}^\infty \frac{\psi(n)a(n)}{n^s}.
\end{align*}
The Hecke relations for $f$ are \cite[\S 2.2]{michel}
\begin{align*}
  a(m)a(d) = \sum_{r|(m,d)}\chi(r)a\!\left(\frac{md}{r^2}\right).
\end{align*}
\begin{lemma}\label{sigma_series} For $f$ and $\psi$ as defined above,
  \begin{align*}
    \sum_{n=1}^\infty \frac{\sigma_v(n,\psi)a(n)}{n^s} = \frac{L(s-v,\psi\times f)L(s,f)}{L(2s-v,\chi\psi)}
  \end{align*}
  for any $s,v \in C$ such that the series on the left converges absolutely.
\end{lemma}
\begin{proof}
  \begin{align*}
    \sum_{n=1}^\infty \frac{\sigma_v(n,\psi)a(n)}{n^s} &= \sum_{n=1}^\infty\sum_{d|n} \psi(d)d^va(n)n^{-s}\\
    &= \sum_{d=1}^\infty\sum_{m=1}^\infty \psi(d)d^v(md)^{-s}a(md).
  \end{align*}
  Consider the expression
  \begin{align*}
    \sum_{d=1}^\infty\sum_{m=1}^\infty \psi(d) d^w m^z a(m)a(d) = \left(\sum_{d=1}^\infty\psi(d)a(d)d^w\right)\left(\sum_{m=1}^\infty a(m)m^z\right).
  \end{align*}
  Using the Hecke relations, this becomes
  \begin{align*}
    \sum_{d=1}^\infty\sum_{m=1}^\infty \psi(d) d^w m^z \sum_{r|(m,d)}\chi(r)a\!\left(\frac{md}{r^2}\right) &= \sum_{r=1}^\infty\sum_{d=1}^\infty\sum_{m=1}^\infty \psi(rd) (rd)^w (rm)^z \chi(r)a(md)\\
    &= \sum_{r=1}^\infty\sum_{d=1}^\infty\sum_{m=1}^\infty \chi(r)\psi(r) r^{w+z} \psi(d)d^w m^z a(md)\\
    &= \left(\sum_{r=1}^\infty \chi(r)\psi(r) r^{w+z}\right)\left(\sum_{d=1}^\infty\sum_{m=1}^\infty \psi(d) d^w m^z a(md)\right)
  \end{align*}
  Rearranging,
  \begin{align*}
    \sum_{d=1}^\infty\sum_{m=1}^\infty \psi(d) d^w m^z a(md) &= \frac{\left(\sum_{d=1}^\infty\psi(d)a(d)d^w\right)\left(\sum_{m=1}^\infty a(m)m^z\right)}{\sum_{r=1}^\infty \chi(r)\psi(r) r^{w+z}}\\
    &= \frac{L(-w,\psi\times f)L(-z,f)}{L(-w-z,\chi\psi)}.
  \end{align*}
  While the series representations that arose in the intermediate steps are only valid in some half-plane, the expression as a ratio of $L$-functions is valid whenever the series its equal to converges by analytic continuation.\\
  \\
  Substituting $w = v-s$ and $z = -s$ yields
  \begin{align*}
    \sum_{d=1}^\infty\sum_{m=1}^\infty \psi(d)d^v(md)^{-s}a(md) = \frac{L(s-v,\psi\times f)L(s,f)}{L(2s-v,\chi\psi)}.
  \end{align*}
\end{proof}

\subsection{The sum of divisors function $\sigma_s(n,\chi)$}
\noindent
The sum of divisors function $\sigma_s(n,\chi)$ is defined as
$$\sigma_s(n,\chi) \coloneqq \sum_{\substack{d\mid n\\d > 0}}\chi(d)d^s.$$
Note that $\sigma_s(n,\chi) = \sigma_s(-n,\chi)$. When we write $\sum_{d|n}$, we will always mean $d > 0$.\\
\\
The following identity is due to Ramanujan. 
\begin{lemma}\label{ramanujan} For $s \in \C$ such that the series below converges absolutely,
  \begin{align*}
    \sum_{n=1}^\infty \frac{\sigma_u(n,\chi)\sigma_v(n,\psi)}{n^s} = \frac{\zeta(s)L(s-u,\chi)L(s-v,\psi)L(s-u-v,\chi\psi)}{L(2s-u-v,\chi\psi)}.
  \end{align*}
\end{lemma}
\noindent
A proof is given for trivial characters and a specific quadratic character in \cite{wilson}, and while the more general version we state here is surely well known, we couldn't find an explicit reference. The proof of \ref{ramanujan} is essentially identical to the one for trivial characters.

\subsection{Eisenstein series}\label{eisenstein_section}
\noindent
Define
\begin{align*}
  \Gamma_0(N) \coloneqq \left\{\begin{pmatrix}a&b\\c&d\end{pmatrix} \in \text{SL}_2(\Z)\,:\, N|c\right\}.
\end{align*}
For a cusp $\fa$ of $\Gamma_0(N)\backslash\HH$, let $\Gamma_\fa$ be the stabilizer of $\fa$ in $\Gamma_0(N)$, and let $\sigma_\fa$ be a matrix in $\text{SL}_2(\R)$ such that $\sigma_\fa\Gamma_\infty\sigma_\fa^{-1} = \Gamma_\fa$ and $\sigma_\fa\infty = \fa$. We will take
\begin{align*}
  \Gamma_\infty = \left\{\begin{pmatrix}1 & x\\0 & 1\end{pmatrix}\,:\,x\in\Z\right\} \quad\text{and}\quad \Gamma_0 = \left\{\begin{pmatrix}1 & 0\\-Nx & 1\end{pmatrix}\,:\,x\in\Z\right\},
\end{align*}
as well as
\begin{align*}
  \sigma_{i\infty} = \begin{pmatrix} 1 & 0 \\ 0 & 1 \end{pmatrix} \quad\text{and}\quad \sigma_0 = N^{-\half}\begin{pmatrix} 0 & -1 \\ N & 0 \end{pmatrix}.
\end{align*}
The matrices $\sigma_\fa$ are called \textit{scaling matrices.} Note that $\sigma_0$ acts as the Fricke involution.\\
\\
Let $N$ be a positive integer and $\chi$ an even Dirichlet character mod $N$. Define the \textit{Eisenstein series} $E_\fa(z,s,\chi)$ by
\begin{align*}
  E_\fa(z,s,\chi) \coloneqq \sum_{\gamma\in\Gamma_\fa\backslash\Gamma_0(N)}\overline{\chi(\gamma)}\text{Im}(\sigma_\fa^{-1}\gamma z)^s.
\end{align*}
We will sometimes call $E_\fa(z,s,\chi)$ the \textit{Eisenstein series attached to the cusp $\fa$}.\\
\\
For the remainder of the section, we will make the following assumptions.
\begin{assumption}\label{young_assumptions}
  ~
  \begin{enumerate}
    \item $N$ is prime,
    \item the weight $k$ appearing in \cite{young} is $0$,
    \item all characters that appear are either the trivial character mod $1$ or a primitive even character mod $N$.
  \end{enumerate}
\end{assumption}
\noindent
We will write $\mathbbm{1}$ to denote the trivial character mod $1$, and we will denote the conductors of $\chi_1$ and $\chi_2$ by $q_1$ and $q_2$ respectively.\\
\\
Young defines the \textit{Eisenstein series attached to the characters $\chi_1$ and $\chi_2$} as follows \cite[(3.3)]{young}.
\begin{align*}
  E_{\chi_1,\chi_2}(z,s) \coloneqq \half \sum_{(c,d)=1} \frac{(q_2y)^s \chi_1(c)\chi_2(d)}{|cq_2z+d|^{2s}}.
\end{align*}
In the definition above, we've taken the weight $k$ in equation \cite[(3.3)]{young} to be $0$, as per \cref{young_assumptions}, which we will refrain from mentioning henceforth. We will not need to know the above definition at any point in what follows, and we include it only for completeness.\\
\\
The Eisenstein series $E_{\chi_1,\chi_2}(z,s)$ is an automorphic function on $\Gamma_0(q_1q_2)$ with nebentypus $\chi_1\overline{\chi_2}$. The completed version of this Eisenstein series is defined by
\begin{align*}
  E_{\chi_1, \chi_2}^*(z,s) \coloneqq \frac{q_2^s}{\tau(\chi_2)\pi^s}\Gamma(s)L(2s,\chi_1\chi_2) E_{\chi_1, \chi_2}(z,s).
\end{align*}
The Eisenstein series attached to cusps can be expressed as follows \cite[(6.2)]{young}.
\begin{align}\label{eisenstein_character_to_cusp}
  E_{i\infty}(z,s,\psi) &= N^{-s}E_{\mathbbm{1},\bar\psi}(z,s) = \frac{\tau(\bar\psi)}{\Lambda(2s,\bar\psi)}E_{\mathbbm{1},\bar\psi}^*(z,s),\\
  E_{0}(z,s,\psi) &= N^{-s}E_{\psi,\mathbbm{1}}(z,s) = \frac{1}{\Lambda(2s,\psi)}E_{\psi,\mathbbm{1}}^*(z,s).\nonumber
\end{align}
The Fourier expansions of these completed Eisenstein series are given below.
\begin{lemma}[{\cite[Thm 4.1]{young}}]\label{eisenstein_series_characters_fourier_expansions}
  \begin{align*}
    &E_{\mathbbm{1},\bar\psi}^*(z,s) = \frac{N^s}{\tau(\bar\psi)}\Lambda(2s,\bar\psi)y^s + \sum_{n\neq 0}2|n|^{\half-s}\sigma_{2s-1}(n,\psi) y^\half K_{s-\half}(2\pi|n|y)e(nx),\\
    &E_{\psi,\mathbbm{1}}^*(z,s) = \frac{N^{1-s}}{\tau(\bar\psi)}\Lambda(2-2s,\bar\psi)y^{1-s} + \sum_{n\neq 0}2|n|^{s-\half}\sigma_{1-2s}(n,\psi) y^\half K_{s-\half}(2\pi|n|y)e(nx).
  \end{align*}
\end{lemma}
\noindent
Consequently,
\begin{lemma}\label{eisenstein_series_cusp_fourier_expansions}
  \begin{align*}
    &E_{i\infty}(z,s,\psi) = y^{s} + \sum_{n\neq 0}\frac{2\tau(\bar\psi)|n|^{\half-s}\sigma_{2s-1}(n,\psi)}{N^s\Lambda(2s,\bar\psi)} y^\half K_{s-\half}(2\pi|n|y)e(nx),\\
    \shortintertext{and}
    &E_{0}(z,s,\psi) = \frac{N^{1-s}}{\tau(\bar\psi)}\frac{\Lambda(2-2s,\bar\psi)}{\Lambda(2s,\psi)} y^{1-s} + \sum_{n\neq 0}\frac{2|n|^{s-\half}\sigma_{1-2s}(n,\psi)}{\Lambda(2s,\psi)} y^\half K_{s-\half}(2\pi|n|y)e(nx).
  \end{align*}
\end{lemma}
\noindent
Eisenstein series attached to characters satisfy the following functional equation \cite[Prop 6.2]{young}
\begin{align*}
  E_{\chi_1,\chi_2}^*(z,s) = E_{\overline{\chi_2},\overline{\chi_1}}^*(z,1-s).
\end{align*}
Moreover, $E_{\chi_1,\chi_2}^*(z,s)$ extends to a meromorphic function for all $s \in \C$, and, if $q_1q_2 > 1$, then it is analytic for all $s \in \C$ \cite[Prop 6.2]{young}.\\
\\
The action of the Fricke involution $\sigma_0$ on $E_{\chi_1,\chi_2}(z,s)$ is given by \cite[\S 9.2]{young}
\begin{align*}
  E_{\chi_1,\chi_2}(\sigma_0 z,s) = E_{\chi_2,\chi_1}(z,s).
\end{align*}
Thus,
\begin{align*}
  &E_{i\infty}(\sigma_0 z, s, \psi) = E_0(z,s,\bar\psi)\\
  \shortintertext{and}
  &E_{\chi_1,\chi_2}^*(\sigma_0 z,s) = \left(\frac{q_2}{q_1}\right)^s \frac{\tau(\chi_1)}{\tau(\chi_2)} E_{\chi_2,\chi_1}^*(z,s).
\end{align*}

\section{$L$-series construction}
\noindent
Our main object of study will be the following function.
\begin{definition}\label{V_def} Let $N$ be a prime and let $\chi$ and $\psi$ be even nontrivial Dirichlet characters mod $N$ such that $\chi\psi$ is nontrivial. Let $u$ and $v$ be nonzero complex numbers such that $|\Re(u)| + |\Re(v)| < \thalf$. Define
  \begin{align*}
    V(z) \coloneqq E_{\mathbbm{1},\bar\chi}^*\!\left(z,\half+u\right) E_{\mathbbm{1},\bar\psi}^*\!\left(z,\half+v\right).
  \end{align*}
\end{definition}
\noindent
See \cref{eisenstein_section} for the definition of the Eisenstein series appearing above.\\
\\
We are interested in the function $V$ because the Fourier coefficients of $V$ can be expressed in terms of $L$-series associated to shifted convolutions of divisor functions. Using the formulas for the Fourier expansions of the Eisenstein series given in \ref{eisenstein_series_characters_fourier_expansions}, we derive the following expressions for the Fourier coefficients of $V(z)$ and $V(\sigma_0 z)$, where $\sigma_0$ is the scaling matrix given in \cref{eisenstein_section}. Below, the function $\Lambda$ denotes the completed Dirichlet $L$-function, defined by \eqref{completed_dirichlet_L_function_eqn}, and $K$ denotes the Bessel function, defined by \eqref{bessel_def}.
\begin{lemma}\label{V_coeffs}
  \begin{alignat*}{2}
    &\int_0^1 V(z)\,dx
    &&= \frac{N^{1+u+v}}{\tau(\bar\chi)\tau(\bar\psi)}\Lambda(1+2u,\bar\chi)\Lambda(1+2v,\bar\psi)y^{1+u+v}\\
    &&&+ \sum_{n=1}^\infty \frac{8\sigma_{2u}(n,\chi)\sigma_{2v}(n\psi)}{n^{u+v}}K_u(2\pi ny)K_v(2\pi ny)y,\\
    &\int_0^1 V(\sigma_0 z)\,dx
    &&= \Lambda(1-2u,\chi)\Lambda(1-2v,\psi)y^{1-u-v}\\
    &&&+ \sum_{n=1}^\infty \frac{N^{1+u+v}}{\tau(\bar\chi)\tau(\bar\psi)} \frac{8\sigma_{-2u}(n,\bar\chi)\sigma_{-2v}(n\bar\psi)}{n^{-u-v}}K_u(2\pi ny)K_v(2\pi ny)y,\\
    &\int_0^1 V(z)\,e(-kx)\,dx
    &&= \frac{N^{\half+u}}{\tau(\bar\chi)}\Lambda(1+2u,\bar\chi)\cdot 2k^{-v} \sigma_{2v}(k,\psi) K_v(2\pi ky) y^{1+u}\\
    &&&+ \frac{N^{\half+v}}{\tau(\bar\psi)}\Lambda(1+2v,\bar\psi)\cdot 2k^{-u} \sigma_{2u}(k,\chi) K_u(2\pi ky) y^{1+v}\\
    &&&+ \sum_{n \neq 0,k}\frac{4\sigma_{2u}(n,\chi)\sigma_{2v}(n-k,\psi)}{|n|^u|n-k|^v}K_u(2\pi|n|y)K_v(2\pi|n-k|y)y.
  \end{alignat*}
  Above, $k$ is an arbitrary positive integer.
\end{lemma}
\noindent
We will need to refer to the individual terms of the $0^{\text{th}}$ Fourier coefficients of $V(z)$ and $V(\sigma_0 z)$. Thus, we make the following definitions.
\begin{definition}\label{V_constant_term_def}
  \begin{align*}
    C_{i\infty}(y) &\coloneqq \sum_{n=1}^\infty \frac{8\sigma_{2u}(n,\chi)\sigma_{2v}(n,\psi)}{n^{u+v}}K_u(2\pi ny)K_v(2\pi ny)y\\
    C_0(y) &\coloneqq \sum_{n=1}^\infty \frac{N^{1+u+v}}{\tau(\bar\chi)\tau(\bar\psi)} \frac{8\sigma_{-2u}(n,\bar\chi)\sigma_{-2v}(n,\bar\psi)}{n^{-u-v}}K_u(2\pi ny)K_v(2\pi ny)y
  \end{align*}
\end{definition}
\noindent
We regard subscripts $i\infty$ and $0$ as cusps of $\Gamma_0(N)\backslash\HH$ because it will be convenient to write ``$C_\fa(y)$'' for an arbitrary cusp $\fa$.
\begin{lemma}\label{C_mellin} For any $s \in \C$ such that $\Re(s) > |\Re(u)| + |\Re(v)|$,
  \begin{align*}
    &\int_0^\infty C_{i\infty}(y) y^{s-1}\frac{dy}{y} = \frac{\Lambda(s+u+v)\Lambda(s-u+v,\chi)\Lambda(s+u-v,\psi)\Lambda(s-u-v,\chi\psi)}{N^{\frac{s-u-v}{2}}\Lambda(2s,\chi\psi)}\\
    &\int_0^\infty C_{0}(y) y^{s-1}\frac{dy}{y} = \frac{N}{\tau(\bar\chi)\tau(\bar\psi)}\frac{\Lambda(s-u-v)\Lambda(s+u-v,\bar\chi)\Lambda(s-u+v,\bar\psi)\Lambda(s+u+v,\overline{\chi\psi})}{N^{\frac{s-u-v}{2}}\Lambda(2s,\overline{\chi\psi})}.
  \end{align*}
\end{lemma}
\begin{lemma}\label{V_coeff_k_mellin} For any positive integer $k$ and any $s\in\C$ such that $\Re(s) > |\Re(u)| + |\Re(v)|$,
  \begin{align*}
    &\int_0^\infty \left(\int_0^1 V(z)\,e(-kx)\,dx\right)y^{s-1}\frac{dy}{y}\\
    &= \frac{\tau(\chi)}{\sqrt{N}}\Lambda(1+2u,\bar\chi)\frac{N^u\sigma_{2v}(k,\psi)}{2\pi^{s+u}k^{s+u+v}}\Gamma\!\left(\frac{s+u+v}{2}\right)\Gamma\!\left(\frac{s+u-v}{2}\right)\\
    &+ \frac{\tau(\psi)}{\sqrt{N}}\Lambda(1+2v,\bar\psi)\frac{N^v\sigma_{2u}(k,\chi)}{2\pi^{s+v}k^{s+u+v}}\Gamma\!\left(\frac{s+u+v}{2}\right)\Gamma\!\left(\frac{s-u+v}{2}\right)\\
    &+ \frac{\Gamma\!\left(\frac{s+u+v}{2}\right)\Gamma\!\left(\frac{s+u-v}{2}\right)\Gamma\!\left(\frac{s-u+v}{2}\right)\Gamma\!\left(\frac{s-u-v}{2}\right)}{2\pi^s\Gamma(s)}\\
    &\hspace{3cm}\cdot\sum_{n\neq 0,k}\frac{\sigma_{2u}(n,\chi)\sigma_{2v}(n-k,\psi)}{|n|^{s+u+v}} {}_2F_1\!\left(\frac{s+u+v}{2},\frac{s+u-v}{2} ;s\,; \frac{2k}{n} - \frac{k^2}{n^2}\right).
  \end{align*}
\end{lemma}
\noindent The shifted convolution $L$-series appearing in the last term is the object we're interested in studying.
\begin{definition}\label{L_def} For the same $\chi, \psi, u, v$ as the ones defining $V$ (see \cref{V_def}), and for any positive integer $k$, define
  \begin{align*}
    L_k(s) \coloneqq \sum_{n\neq 0,k}\frac{\sigma_{2u}(n,\chi)\sigma_{2v}(n-k,\psi)}{|n|^{s+u+v}} {}_2F_1\!\left(\frac{s+u+v}{2},\frac{s+u-v}{2} ;s\,; \frac{2k}{n} - \frac{k^2}{n^2}\right).
  \end{align*}
\end{definition}
\noindent
In the definition above ${}_2F_1$ is the Gaussian hypergeometric function.\\
\\
This type of construction forms the basis of a very well known and widely used approach for studying shifted convolutions. The specific construction that we're considering is particularly similar to work of Takhtadzhyan--Vinogradov and Jutila. We highlight the papers \cite{TV, jutila}, but those authors have written many papers on this topic. However, we remark that there is a small technical difference between the approach in \cite{TV, jutila} and the one here. Those papers access shifted convolutions by considering an inner product with the Poincar\'{e} series
\begin{align*}
  P_k(z,s) \coloneqq \sum_{\gamma\in\Gamma_\infty\backslash\Gamma_0(N)}e(k\gamma z)\text{Im}(\gamma z)^s,
\end{align*}
while we directly take the Mellin transform of a Fourier coefficient in a way that's very similar to what's done in \cite{goldfeld}. If $f(z) = \sum_n \alpha_n(y)e(nx)$ is an automorphic function, then
\begin{align*}
  &\left\langle f, P_k(\cdot,\bar s) \right\rangle = \int_0^\infty \alpha_k(y)e^{-2\pi ky}y^{s-1}\frac{dy}{y}\\
  \shortintertext{while}
  &\int_0^\infty\left(\int_0^1f(z)\,e(-kx)\,dx\right)y^{s-1}\frac{dy}{y} = \int_0^\infty \alpha_k(y)y^{s-1}\frac{dy}{y},
\end{align*}
i.e.\ we avoid a factor of $e^{-2\pi k y}$ in the Mellin transform. For applications to the additive divisor problem, this factor ends up being very close to $1$ in the relevant range, roughly speaking, and can be dealt with via analysis which is in principle routine but in practice quite tedious. The method we use here requires that we evaluate the Mellin transform of the function $K_u(ay)K_v(by)$ (given in \ref{bessel_product_mellin}), which is used frequently and typically relatively straightforward to handle. In contrast, taking an inner product with $P_k(z,s)$ requires one to evaluate the Mellin transform of $K_u(ay)K_v(by)e^{-cy}$, which seems substantially more mysterious and unwieldy. Sizable parts of the papers \cite{TV, jutila, JM} are devoted to handling this Mellin transform, and each uses a different approach. We note that a formula for this Mellin transform is given in \cite[3.14.19.1]{BMS} for generic $u,v$, and it can be manipulated further to obtain an expression for the case $u,v = 0$ that arises in the classical additive divisor problem. However, we haven't seen this formula used, and it's unclear to us how helpful it'd be.

\section{Automorphic regularization}
\noindent
We now look to express the Fourier coefficients of $V$ in a way that's distinct from the one in \ref{V_coeff_k_mellin}, since doing so will yield a nontrivial expressions for the shifted convolution $L_k$ (\cref{L_def}). Observe that $V$ is an automorphic function of level $N$ and character $\chi\psi$. If $V$ were in $L^2(\Gamma_0(N)\backslash\HH\,;\chi\psi)$, then we could produce an expression for its Fourier coefficients via spectral decomposition (see \cite[\S 2.1.2.1]{michel}, e.g.).
\begin{theorem}[Spectral decomposition]\label{spectral_decomposition} Suppose $\{f_j\}$ is a set of Maass forms which forms a complete orthogonal basis for the cuspidal subspace of $L^2(\Gamma_0(N)\backslash\HH\,;\chi\psi)$. Suppose that $U \in L^2(\Gamma_0(N)\backslash\HH\,;\chi\psi)$. Then
  \begin{align*}
    U(z) = \sum_j \frac{\left\langle U, f_j\right\rangle}{\langle f_j, f_j\rangle}f_j(z) + \frac{1}{4\pi i}\sum_\fa \int_{\left(\thalf\right)}\left\langle U, E_\fa(\cdot, \lambda, \chi\psi)\right\rangle E_\fa(z,\lambda,\chi\psi)\,d\lambda.
  \end{align*}
\end{theorem}
\noindent
However $V \not\in L^2(\Gamma_0(N)\backslash\HH\,;\chi\psi)$, and so we can't use \cref{spectral_decomposition} directly. Instead, our strategy will be to find an automorphic function $R$ such that $V - R \in L^2(\Gamma_0(N)\backslash\HH\,;\chi\psi)$, and such that we understand the Fourier coefficients of $R$ by other means. We call this the method of \textit{automorphic regularization}. In \cite{zagier}, Zagier discusses automorphic regularization for level $1$ and trivial character, and, in section 3, demonstrates the method using a product of Eisenstein series very similar to $V$. In \cite{templier}, Templier gives an explicit automorphic regularization for an analogue of $V$ with prime level and trivial character.\\
\\
We will find that, to regularize $V$, the following choice of $R$ will be suitable:
\begin{definition}\label{R_def} For the same $\chi, \psi, u, v$ as the ones defining $V$ (see \cref{V_def}), define
  \begin{align*}
    R(z) \coloneqq \frac{\tau(\overline{\chi\psi})}{\tau(\bar\chi)\tau(\bar\psi)}\frac{\Lambda(1+2u,\bar\chi)\Lambda(1+2v,\bar\psi)}{\Lambda(2+2u+2v,\overline{\chi\psi})} &E_{\mathbbm{1},\overline{\chi\psi}}^*(z,1+u+v)\\
    +\,\, \frac{\Lambda(1-2u,\chi)\Lambda(1-2v,\psi)}{\Lambda(2-2u-2v,\chi\psi)} &E_{\mathbbm{1},\overline{\chi\psi}}^*(z,u+v).
  \end{align*}
\end{definition}
\noindent
See \cref{eisenstein_section} for the definition of the Eisenstein series $E^*$ appearing above, and \eqref{completed_dirichlet_L_function_eqn} for the definition of the completed $L$-function $\Lambda$.
\begin{lemma}\label{R_coeffs}
  For any positive integer $k$,
  \begin{align*}
    \int_0^1 R(z)\,e(-kx)\,dx
    = \frac{\tau(\overline{\chi\psi})}{\tau(\bar\chi)\tau(\bar\psi)} &\frac{\Lambda(1+2u,\bar\chi)\Lambda(1+2v,\bar\psi)}{\Lambda(2+2u+2v,\overline{\chi\psi})}\cdot 2k^{-u-v-\half} \sigma_{2u+2v+1}(k,\chi\psi) K_{u+v+\half}(2\pi ky) y^{\half}\\
    +\,\, &\frac{\Lambda(1-2u,\chi)\Lambda(1-2v,\psi)}{\Lambda(2-2u-2v,\chi\psi)}\cdot 2k^{-u-v+\half}\sigma_{2u+2v-1}(k,\chi\psi) K_{u+v-\half}(2\pi ky) y^{\half}.
  \end{align*}
\end{lemma}
\begin{lemma}\label{R_coeff_k_mellin} For any positive integer $k$ and any $s \in \C$ such that $\Re(s) > 1 + |\Re(u+v)|$,
  \begin{align*}
    \int_0^\infty\left(\int_0^1 R(z)\,e(-kx)\,dx\right)y^{s-1}\frac{dy}{y}&\\
    = \frac{\tau(\overline{\chi\psi})}{\tau(\bar\chi)\tau(\bar\psi)} &\frac{\Lambda(1+2u,\bar\chi)\Lambda(1+2v,\bar\psi)}{\Lambda(2+2u+2v,\overline{\chi\psi})} \frac{\sigma_{2u+2v+1}(k,\chi\psi)}{2\pi^{s-\half}k^{s+u+v}} \Gamma\!\left(\frac{s+u+v}{2}\right)\Gamma\!\left(\frac{s-1-u-v}{2}\right)\\
    +\,\, &\frac{\Lambda(1-2u,\chi)\Lambda(1-2v,\psi)}{\Lambda(2-2u-2v,\chi\psi)} \frac{\sigma_{2u+2v-1}(k,\chi\psi)}{2\pi^{s-\half}k^{s-1+u+v}} \Gamma\!\left(\frac{s-u-v}{2}\right)\Gamma\!\left(\frac{s-1+u+v}{2}\right).
  \end{align*}
\end{lemma}
\begin{lemma}\label{V_regularization}
  \begin{align*}
    V - R \in L^2(\Gamma_0(N)\backslash\HH\,;\chi\psi).
  \end{align*}
\end{lemma}
\begin{proof}
  From the expressions \ref{V_coeffs} and \ref{eisenstein_series_characters_fourier_expansions}, and the fact that $K$-bessel functions decay exponentially \cite[8.451.6]{GR}, we see that, as $y \to \infty$,
  \begin{align}\label{growth_eqn}
    V(z) - R(z) \ll y^{|\Re(u+v)|} \quad\quad\text{and}\quad\quad V(\sigma_0^{-1} z) - R(\sigma_0^{-1} z) \ll y^{|\Re(u+v)|}.
  \end{align}
  The $L^2$ norm of $V - R$ is given by the integral
  \begin{align}\label{L2_norm_eqn}
    \int_{\Gamma_0(N)\backslash\HH}\left|V(z) - R(z)\right|^2 \frac{dxdy}{y^2}.
  \end{align}
  One can take a fundamental domain for $\Gamma_0(N)\backslash\HH$ which consists of the union of translates of the standard full level fundamental domain $\{z\in\HH\,:\, |x| \leq \thalf, |z| \geq 1\}$ by the coset representatives of $\Gamma_0(N)$ in $\text{SL}_2(\Z)$. See \cite[\S 3]{zagier2} for an explicit construction in the case of prime level, and \cite[\S 2.4]{iwaniec_topics} for discussion of general level. For prime $N$, there are $N+1$ such translates, $1$ of which touches the cusp $i\infty$ and has $y$ bounded away from $0$, and $N$ of which touch the cusp $0$ and have $y$ bounded away from $i\infty$ \cite[(2.29)]{iwaniec_topics}. In each of these translates, we make the change of variables $z \mapsto \sigma_\fa^{-1} z$, where $\fa$ is the cusp the translate contains. Recall that it was assumed in the definition \ref{V_def} of $V$ that $|\Re(u+v)| < \half$. It then follows from (\ref{growth_eqn}) and the preceding observations about the geometry of $\Gamma_0(N)\backslash\HH$ that the quantity (\ref{L2_norm_eqn}) is finite.
\end{proof}
\noindent
One surprising feature of our automorphic regularization is the following:
\begin{lemma}\label{regularized_inner_product} For any $s \in \C$ satisfying $|\Re(u+v)| < \Re(s) < 1 - |\Re(u+v)|$ as well as $\Lambda(2s,\chi\psi) \neq 0$ and $\Lambda(2s,\overline{\chi\psi}) \neq 0$, we have
  \begin{align*}
    \left\langle V - R, E_\fa(\cdot, \bar s, \chi\psi) \right\rangle = \int_0^\infty C_\fa(y)y^{s-1}\frac{dy}{y},
  \end{align*}
  where $C_\fa(y)$ is the part of the constant term of $V$ defined in \ref{V_constant_term_def}.
\end{lemma}
\begin{proof}
  This is proposition A.3 of \cite{HKLDW}. In our case the scattering matrix has a pole exactly when $\Lambda(2s,\chi\psi) = 0$ or $\Lambda(2s,\overline{\chi\psi}) = 0$.
\end{proof}
\noindent
We call the result \cref{regularized_inner_product} surprising because the constant term of the difference $V(z) - R(z)$ has a term proportional to $y^{u+v}$, coming from the Eisenstein series in $R(z)$, and one might expect that term to cause the inner product in \ref{regularized_inner_product} to diverge. As remarked in \cite[\S 3]{zagier}, \ref{regularized_inner_product} can be viewed as a generalization of the Maass-Selberg relations \cite[Prop. 6.8]{iwaniec}. Moreover, it is a general phenomenon that, when regularizing in this way, terms of the form $y^\alpha$ with $\Re(\alpha) < \thalf$ do not contribute to the inner products that appear in spectral expansions; see for example \cite[\S 4]{MV}, \cite[\S 2]{wu}, and \cite[\S 8]{nelson}.\\
\\
We can now proceed with our study of the shifted convolution $L_k$ by computing the spectral decomposition of $V - R$. Ultimately what we will find is that the behaviour of $L_k$ is controlled primarily by the regularizing function $R$.

\section{Spectral decomposition}
\noindent
In this section we will use the spectral decomposition
\begin{align*}
  V(z) - R(z) = \sum_j \frac{\langle V - R, f_j\rangle}{\langle f_j, f_j\rangle}f_j(z) + \frac{1}{4\pi i}\sum_\fa \int_{\left(\thalf\right)}\langle V - R, E_\fa(\cdot, \lambda, \chi\psi)\rangle E_\fa(z,\lambda,\chi\psi)\,d\lambda
\end{align*}
to write the Fourier coefficients of $V$ in a way that's different from the one given in \ref{V_coeffs}. We'll do this by evaluating the inner product in the spectral decomposition above, and then referencing the known Fourier expansions of the functions $f_j$ and $E_\fa$. Producing two distinct expressions for the $k^\tth$ Fourier coefficient of $V$ will lead to nontrivial identities for the shifted convolution $L_k$ (\cref{L_def}) that appeared when taking the Mellin transform in \ref{V_coeff_k_mellin}.\\
\\
We first determine the contribution from the continuous spectrum.
\begin{lemma}\label{continuous_spectrum_contribution} For any positive integer $k$ and any $s \in \C$ with $\Re(s) > \half$,
  \begin{align*}
    &\int_0^\infty \int_0^1 \left[\frac{1}{4\pi i}\sum_\fa \int_{\left(\thalf\right)}\langle V - R, E_\fa(\cdot, \lambda, \chi\psi)\rangle E_\fa(z,\lambda,\chi\psi)\,d\lambda \right] e(-kx)y^{s-1}\frac{dxdy}{y}\\
    &= \int_{-\infty}^\infty \frac{\Lambda(\half + i\omega + u + v)\Lambda(\half + i\omega - u + v,\chi)\Lambda(\half + i\omega + u - v,\psi)\Lambda(\half + i\omega - u - v,\chi\psi)}{\Lambda(1+2i\omega,\chi\psi)\Lambda(1-2i\omega,\overline{\chi\psi})}\\
    &\hspace{2cm}\cdot\frac{\tau(\overline{\chi\psi})}{\sqrt{N}}\frac{\sigma_{-2i\omega}(k,\chi\psi)}{4N^{\frac{\half-i\omega-u-v}{2}} \pi^{s+\half} k^{s-\half - i\omega}} \Gamma\!\left(\frac{s - \half + i\omega}{2}\right)\Gamma\!\left(\frac{s - \half - i\omega}{2}\right) d\omega.
  \end{align*}
\end{lemma}
\begin{proof}
Recall \cref{regularized_inner_product}, which says that
\begin{align*}
  \left\langle V - R, E_\fa(\cdot, \bar s, \chi\psi) \right\rangle = \int_0^\infty C_\fa(y)y^{s-1}\frac{dy}{y},
\end{align*}
where $C_\fa(y)$ is the part of the constant term of $V$ (\cref{V_constant_term_def}). This identity is the main benefit of the process of automorphic regularization, as it allows us to ignore the part of the constant term of $V - R$ which is not of rapid decay. We evaluate the Mellin transform on the right hand side of the equation above for both cusps using in \ref{C_mellin}. Thus, for $s \in \C$ such that $\Re(s) > |\Re(u)| + |\Re(v)|$, we have
\begin{align*}
  &\left\langle V - R, E_{i\infty}(\cdot, \bar s, \chi\psi) \right\rangle = \frac{\Lambda(s+u+v)\Lambda(s-u+v,\chi)\Lambda(s+u-v,\psi)\Lambda(s-u-v,\chi\psi)}{N^{\frac{s-u-v}{2}}\Lambda(2s,\chi\psi)}\\
  \shortintertext{and}
  &\left\langle V - R, E_{0}(\cdot, \bar s, \chi\psi) \right\rangle = \frac{N}{\tau(\bar\chi)\tau(\bar\psi)}\frac{\Lambda(s-u-v)\Lambda(s+u-v,\bar\chi)\Lambda(s-u+v,\bar\psi)\Lambda(s+u+v,\overline{\chi\psi})}{N^{\frac{s-u-v}{2}}\Lambda(2s,\overline{\chi\psi})}.
\end{align*}
From \ref{eisenstein_series_cusp_fourier_expansions}, we have
\begin{align*}
  &\int_0^1 E_{i\infty}(z,s,\chi\psi)\,e(-kx)\,dx = \frac{2\tau(\overline{\chi\psi})k^{\half-s}\sigma_{2s-1}(k,\chi\psi)}{N^s\Lambda(2s,\overline{\chi\psi})} y^\half K_{s-\half}(2\pi ky)\\
  \shortintertext{and}
  &\int_0^1 E_{0}(z,s,\chi\psi)\,e(-kx)\,dx = \frac{2k^{s-\half}\sigma_{1-2s}(k,\chi\psi)}{\Lambda(2s,\chi\psi)} y^\half K_{s-\half}(2\pi ky).
\end{align*}
Substituting these expressions into
\begin{align*}
  \int_0^1 \left[\frac{1}{4\pi i}\sum_{\fa \in \{i\infty, 0\}} \int_{\left(\thalf\right)}\langle V - R, E_\fa(\cdot, \lambda, \chi\psi)\rangle E_\fa(z,\lambda,\chi\psi)\,d\lambda \right] e(-kx)\,dx,
\end{align*}
making the change of variables $\lambda \mapsto \thalf + i\omega$, and using the functional equation for the completed $L$-functions in the $\fa = 0$, we obtain
\begin{align*}
  &\int_0^1 \left[\frac{1}{4\pi i}\sum_{\fa \in \{i\infty, 0\}} \int_{\left(\thalf\right)}\langle V - R, E_\fa(\cdot, \lambda, \chi\psi)\rangle E_\fa(z,\lambda,\chi\psi)\,d\lambda \right] e(-kx)\,dx\\
  &= \int_{-\infty}^\infty \frac{\Lambda(\half + i\omega + u + v)\Lambda(\half + i\omega - u + v,\chi)\Lambda(\half + i\omega + u - v,\psi)\Lambda(\half + i\omega - u - v,\chi\psi)}{\Lambda(1+2i\omega,\chi\psi)\Lambda(1-2i\omega,\overline{\chi\psi})}\\
  &\hspace{2cm}\cdot\frac{\tau(\overline{\chi\psi})}{\sqrt{N}}\frac{k^{i\omega} \sigma_{-2i\omega}(k,\chi\psi)}{\pi N^{\frac{\half-i\omega-u-v}{2}}}K_{i\omega}(2\pi k y)y^\half d\omega.
\end{align*}
We note that the two terms corresponding to $\fa = i\infty$ and $\fa = 0$ are equal. \ref{continuous_spectrum_contribution} then follows from \ref{bessel_mellin}.
\end{proof}
\noindent
Now we consider the contribution from the discrete spectrum. Let $f_j$ be a Maass eigenform with Laplacian eigenvalue $\tfrac{1}{4} + r_j^2$. The Fourier expansion $f_j$ can be written, for some $\rho_j(n) \in \C$, as
\begin{align}\label{maass_form_fourier_expansion}
  f_j(z) = \sum_{n\neq 0}2\rho_j(n)y^\half K_{ir_j}(2\pi|n|y)e(nx).
\end{align}
Moreover, for any $n > 0$, we have $\rho_j(n) = \rho_j(1)a_j(n)$ and $\rho_j(-n) = \rho_j(-1)a_j(n)$, where $a_j(n)$ is the Hecke-eigenvalue of $f_j$ \cite[\S 2.2]{michel}. If $f_j$ is even (i.e. $f_j(z) = f_j(-\bar z)$), then $\rho(1) = \rho(-1)$, while if $f_j$ is odd (i.e. $f_j(z) = -f_j(-\bar z)$), then $\rho(1) = -\rho(-1)$. Define
\begin{align*}
  &L(s,\overline{f_j}) \coloneqq \sum_{n=1}^\infty \frac{\overline{a_j(n)}}{n^s}\\
  \shortintertext{and}
  &L(s,\psi\times\overline{f_j}) \coloneqq \sum_{n=1}^\infty \frac{\psi(n)\overline{a_j(n)}}{n^s}.
\end{align*}
\begin{lemma}\label{discrete_spectrum_contribution} For any positive integer $k$ and any $s \in \C$ with $\Re(s) > \half$,
  \begin{align*}
    &\int_0^\infty \int_0^1 \left[ \sum_j \frac{\langle V - R, f_j\rangle}{\langle f_j, f_j\rangle}f_j(z) \right] e(-kx)y^{s-1}\frac{dxdy}{y}\\
    &= \sum_{f_j \text{ even}} \frac{\tau(\chi)}{\sqrt{N}}\frac{N^{u}|\rho_j(1)|^2 a_j(k)}{2\pi^{s+\half+2u} k^{s-\half}\langle f_j, f_j \rangle}L(\thalf+u+v,\overline{f_j})L(\thalf+u-v,\psi\times\overline{f_j})\\
    &\hspace{1cm}\cdot\Gamma\!\left(\frac{u+v+\half+ir_j}{2}\right)\Gamma\!\left(\frac{u-v+\half+ir_j}{2}\right)\Gamma\!\left(\frac{u+v+\half-ir_j}{2}\right)\Gamma\!\left(\frac{u-v+\half-ir_j}{2}\right)\\
    &\hspace{1cm}\cdot\Gamma\!\left(\frac{s - \half + ir_j}{2}\right)\Gamma\!\left(\frac{s - \half - ir_j}{2}\right).
  \end{align*}  
\end{lemma}
\noindent
To prove \ref{discrete_spectrum_contribution}, we will first need to evaluate the inner product $\langle V - R, f_j \rangle$. The space of cusp forms is orthogonal to the space of Eisenstein series, so
\begin{align*}
  \langle V - R, f_j \rangle = \langle V, f_j \rangle.
\end{align*}
Using (\ref{eisenstein_character_to_cusp}), we can write
\begin{align*}
  V(z) = \frac{\Lambda(1+2u,\bar\chi)}{\tau(\bar\chi)}E_{i\infty}(z,\thalf+u,\chi)E_{\mathbbm{1},\bar\psi}^*(z,\thalf+v).
\end{align*}
This lets us compute the inner product $\langle V, f_j \rangle$ by unfolding the factor $E_{i\infty}(z,\thalf+u,\chi)$. We use \ref{bessel_product_mellin} to evaluate the integral over $y$, and \ref{sigma_series} to evaluate the resulting sum.
\begin{lemma}\label{discrete_spectrum_inner_product} If $f_j$ is even, then
  \begin{align*}
    \langle V - R, f_j \rangle &= \frac{N^{\half+u}\overline{\rho_j(1)}}{\tau(\bar\chi)\pi^{1+2u}}L(\thalf+u+v,\overline{f_j})L(\thalf+u-v,\psi\times\overline{f_j})\\
    &\hspace{1cm}\cdot\Gamma\!\left(\frac{\half+u+v+ir_j}{2}\right)\Gamma\!\left(\frac{\half+u-v+ir_j}{2}\right)\Gamma\!\left(\frac{\half+u+v-ir_j}{2}\right)\Gamma\!\left(\frac{\half+u-v-ir_j}{2}\right).
  \end{align*}
  If $f_j$ is odd, then $\langle V - R, f_j \rangle = 0$.
\end{lemma}
\noindent
\begin{proof}[Proof of \ref{discrete_spectrum_contribution}]
Using \ref{discrete_spectrum_inner_product} and (\ref{maass_form_fourier_expansion}),
\begin{align*}
  &\int_0^1 \left[ \sum_j \frac{\langle V - R, f_j\rangle}{\langle f_j, f_j\rangle}f_j(z) \right] e(-kx)\,dx\\
  &= \sum_{f_j \text{ even}} \frac{2N^{\half+u}|\rho_j(1)|^2 a_j(k)}{\tau(\bar\chi)\pi^{1+2u}\langle f_j, f_j \rangle}L(\thalf+u+v,\overline{f_j})L(\thalf+u-v,\psi\times\overline{f_j})\\
  &\hspace{1cm}\cdot\Gamma\!\left(\frac{u+v+\half+ir_j}{2}\right)\Gamma\!\left(\frac{u-v+\half+ir_j}{2}\right)\Gamma\!\left(\frac{u+v+\half-ir_j}{2}\right)\Gamma\!\left(\frac{u-v+\half-ir_j}{2}\right)K_{ir_j}(2\pi ky) y^\half.
\end{align*}
Taking the Mellin transform and using \ref{bessel_mellin} yields the right hand side of \ref{discrete_spectrum_contribution}.
\end{proof}
\noindent
By combining \ref{continuous_spectrum_contribution}, \ref{discrete_spectrum_contribution}, and \ref{R_coeff_k_mellin}, we arrive at the following identity involving the shifted convolution $L_k$.
\begin{proposition}\label{L_identity}
  For any positive integer $k$ and and $s \in \C$ such that $\Re(s) > 1 + |\Re(u)| + |\Re(v)|$,
  \begin{align*}
    &\frac{\Gamma\!\left(\frac{s+u+v}{2}\right)\Gamma\!\left(\frac{s+u-v}{2}\right)\Gamma\!\left(\frac{s-u+v}{2}\right)\Gamma\!\left(\frac{s-u-v}{2}\right)}{2\pi^s\Gamma(s)} L_k(s)\\
    &= \frac{\tau(\overline{\chi\psi})}{\tau(\bar\chi)\tau(\bar\psi)}\frac{\Lambda(1+2u,\bar\chi)\Lambda(1+2v,\bar\psi)}{\Lambda(2+2u+2v,\overline{\chi\psi})} \frac{\sigma_{2u+2v+1}(k,\chi\psi)}{2\pi^{s-\half}k^{s+u+v}} \Gamma\!\left(\frac{s+u+v}{2}\right)\Gamma\!\left(\frac{s-1-u-v}{2}\right)\\
    &+\frac{\Lambda(1-2u,\chi)\Lambda(1-2v,\psi)}{\Lambda(2-2u-2v,\chi\psi)} \frac{\sigma_{2u+2v-1}(k,\chi\psi)}{2\pi^{s-\half}k^{s-1+u+v}} \Gamma\!\left(\frac{s-u-v}{2}\right)\Gamma\!\left(\frac{s-1+u+v}{2}\right)\\
    &- \frac{\tau(\chi)}{\sqrt{N}}\Lambda(1+2u,\bar\chi)\frac{N^u\sigma_{2v}(k,\psi)}{2\pi^{s+u}k^{s+u+v}}\Gamma\!\left(\frac{s+u+v}{2}\right)\Gamma\!\left(\frac{s+u-v}{2}\right)\\
    &- \frac{\tau(\psi)}{\sqrt{N}}\Lambda(1+2v,\bar\psi)\frac{N^v\sigma_{2u}(k,\chi)}{2\pi^{s+v}k^{s+u+v}}\Gamma\!\left(\frac{s+u+v}{2}\right)\Gamma\!\left(\frac{s-u+v}{2}\right)\\
    &+ \int_{-\infty}^\infty \frac{\Lambda(\half + i\omega + u + v)\Lambda(\half + i\omega - u + v,\chi)\Lambda(\half + i\omega + u - v,\psi)\Lambda(\half + i\omega - u - v,\chi\psi)}{\Lambda(1+2i\omega,\chi\psi)\Lambda(1-2i\omega,\overline{\chi\psi})}\\
    &\hspace{2cm}\cdot\frac{\tau(\overline{\chi\psi})}{\sqrt{N}}\frac{\sigma_{-2i\omega}(k,\chi\psi)}{4N^{\frac{\half-i\omega-u-v}{2}} \pi^{s+\half} k^{s-\half - i\omega}} \Gamma\!\left(\frac{s - \half + i\omega}{2}\right)\Gamma\!\left(\frac{s - \half - i\omega}{2}\right) d\omega\\
    &+ \sum_{f_j \text{ even}} \frac{\tau(\chi)}{\sqrt{N}}\frac{N^{u}|\rho_j(1)|^2 a_j(k)}{2\pi^{s+\half+2u} k^{s-\half}\langle f_j, f_j \rangle}L(\thalf+u+v,\overline{f_j})L(\thalf+u-v,\psi\times\overline{f_j})\\
    &\hspace{1cm}\cdot\Gamma\!\left(\frac{u+v+\half+ir_j}{2}\right)\Gamma\!\left(\frac{u-v+\half+ir_j}{2}\right)\Gamma\!\left(\frac{u+v+\half-ir_j}{2}\right)\Gamma\!\left(\frac{u-v+\half-ir_j}{2}\right)\\
    &\hspace{1cm}\cdot\Gamma\!\left(\frac{s - \half + ir_j}{2}\right)\Gamma\!\left(\frac{s - \half - ir_j}{2}\right).
  \end{align*}
\end{proposition}
\noindent
We'll use the identity above to study $L_k$. It'll be convenient to break the right hand side up into four parts: the contribution from the continuous spectrum (\ref{continuous_spectrum_contribution}), the contribution from the discrete spectrum (\ref{discrete_spectrum_contribution}), the contribution from $R$ (\ref{R_coeff_k_mellin}), and the contribution from the polynomial part of the constant term of $V$ (the first two terms of \ref{V_coeff_k_mellin}).
\begin{definition}\label{L_parts_def}
  \begin{align*}
    &L_k^{(\text{cont})}(s) \coloneqq \int_{-\infty}^\infty \frac{\Lambda(\half + i\omega + u + v)\Lambda(\half + i\omega - u + v,\chi)\Lambda(\half + i\omega + u - v,\psi)\Lambda(\half + i\omega - u - v,\chi\psi)}{\Lambda(1+2i\omega,\chi\psi)\Lambda(1-2i\omega,\overline{\chi\psi})}\\
    &\hspace{4cm}\cdot\frac{\tau(\overline{\chi\psi})}{\sqrt{N}}\frac{\sigma_{-2i\omega}(k,\chi\psi)}{2\pi^{\half}N^{\frac{\half-i\omega-u-v}{2}} k^{s-\half - i\omega}} \frac{\Gamma\!\left(\frac{s - \half + i\omega}{2}\right)\Gamma\!\left(\frac{s - \half - i\omega}{2}\right)\Gamma(s)}{\Gamma\!\left(\frac{s+u+v}{2}\right)\Gamma\!\left(\frac{s+u-v}{2}\right)\Gamma\!\left(\frac{s-u+v}{2}\right)\Gamma\!\left(\frac{s-u-v}{2}\right)} d\omega\\
    &L_k^{(\text{disc})}(s) \coloneqq \sum_{f_j \text{ even}} \frac{\tau(\chi)}{\sqrt{N}}\frac{N^{u}|\rho_j(1)|^2 a_j(k)}{\pi^{\half+2u} k^{s-\half}\langle f_j, f_j \rangle}L(\thalf+u+v,\overline{f_j})L(\thalf+u-v,\psi\times\overline{f_j})\\
    &\hspace{2.5cm}\cdot\frac{\Gamma\!\left(\frac{u+v+\half+ir_j}{2}\right)\Gamma\!\left(\frac{u-v+\half+ir_j}{2}\right)\Gamma\!\left(\frac{u+v+\half-ir_j}{2}\right)\Gamma\!\left(\frac{u-v+\half-ir_j}{2}\right)\Gamma\!\left(\frac{s - \half + ir_j}{2}\right)\Gamma\!\left(\frac{s - \half - ir_j}{2}\right)\Gamma(s)}{\Gamma\!\left(\frac{s+u+v}{2}\right)\Gamma\!\left(\frac{s+u-v}{2}\right)\Gamma\!\left(\frac{s-u+v}{2}\right)\Gamma\!\left(\frac{s-u-v}{2}\right)}\\
    &L_k^{(R)}(s) \coloneqq \frac{\tau(\overline{\chi\psi})}{\tau(\bar\chi)\tau(\bar\psi)}\frac{\Lambda(1+2u,\bar\chi)\Lambda(1+2v,\bar\psi)}{\Lambda(2+2u+2v,\overline{\chi\psi})} \frac{\pi^{\half}\sigma_{2u+2v+1}(k,\chi\psi)}{k^{s+u+v}} \frac{\Gamma\!\left(\frac{s+u+v}{2}\right)\Gamma\!\left(\frac{s-1-u-v}{2}\right)\Gamma(s)}{\Gamma\!\left(\frac{s+u+v}{2}\right)\Gamma\!\left(\frac{s+u-v}{2}\right)\Gamma\!\left(\frac{s-u+v}{2}\right)\Gamma\!\left(\frac{s-u-v}{2}\right)}\\
    &\hspace{2.7cm}+\frac{\Lambda(1-2u,\chi)\Lambda(1-2v,\psi)}{\Lambda(2-2u-2v,\chi\psi)} \frac{\pi^{\half}\sigma_{2u+2v-1}(k,\chi\psi)}{k^{s-1+u+v}} \frac{\Gamma\!\left(\frac{s-u-v}{2}\right)\Gamma\!\left(\frac{s-1+u+v}{2}\right)\Gamma(s)}{\Gamma\!\left(\frac{s+u+v}{2}\right)\Gamma\!\left(\frac{s+u-v}{2}\right)\Gamma\!\left(\frac{s-u+v}{2}\right)\Gamma\!\left(\frac{s-u-v}{2}\right)}\\
    &L_k^{(V)}(s) \coloneqq - \frac{\tau(\chi)}{\sqrt{N}}\Lambda(1+2u,\bar\chi)\frac{N^u\sigma_{2v}(k,\psi)}{\pi^{u}k^{s+u+v}}\frac{\Gamma\!\left(\frac{s+u+v}{2}\right)\Gamma\!\left(\frac{s+u-v}{2}\right)\Gamma(s)}{\Gamma\!\left(\frac{s+u+v}{2}\right)\Gamma\!\left(\frac{s+u-v}{2}\right)\Gamma\!\left(\frac{s-u+v}{2}\right)\Gamma\!\left(\frac{s-u-v}{2}\right)}\\
    &\hspace{1.6cm}-\! \frac{\tau(\psi)}{\sqrt{N}}\Lambda(1+2v,\bar\psi)\frac{N^v\sigma_{2u}(k,\chi)}{\pi^{v}k^{s+u+v}}\frac{\Gamma\!\left(\frac{s+u+v}{2}\right)\Gamma\!\left(\frac{s-u+v}{2}\right)\Gamma(s)}{\Gamma\!\left(\frac{s+u+v}{2}\right)\Gamma\!\left(\frac{s+u-v}{2}\right)\Gamma\!\left(\frac{s-u+v}{2}\right)\Gamma\!\left(\frac{s-u-v}{2}\right)}
  \end{align*}
\end{definition}
\noindent
The definitions above are such that
\begin{align*}
  L_k(s) = L_k^{(\text{cont})}(s) + L_k^{(\text{disc})}(s) + L_k^{(R)}(s) + L_k^{(V)}(s).
\end{align*}

\section{Growth on vertical lines}\label{growth_section}
\noindent
Let $L_k^{(R)}(s)$, $L_k^{(R)}(s)$, $L_k^{(\text{cont})}(s)$, and $L_k^{(\text{disc})}(s)$ be as in \cref{L_parts_def}. Let $\sigma$ be such that the aforementioned functions are holomorphic on the vertical line $\Re(s) = \sigma$. \cref{L_growth_R}, \ref{L_growth_V}, \ref{L_growth_cont}, and \ref{L_growth_disc} as $|t| \to \infty$, with the implied constants depending only on $N$, $u$, $v$, $k$, $\sigma$, and $\eps$.

\begin{lemma}\label{L_growth_R} 
  \begin{flalign*}
    \hspace{4cm} L_k^{(R)}(s) \ll 1. &&
  \end{flalign*}
\end{lemma}

\begin{lemma}\label{L_growth_V} 
  \begin{flalign*}
    \hspace{4cm} L_k^{(V)}(s) \ll |t|^{\half + \Re(u)} + |t|^{\half + \Re(v)}. &&
  \end{flalign*}
\end{lemma}

\begin{lemma}\label{L_growth_cont} 
  \begin{flalign*}
    \hspace{4cm} L_k^{(\text{\emph{cont}})}(s) \ll |t|^{1+\eps}. &&
  \end{flalign*}
\end{lemma}

\begin{lemma}\label{L_growth_disc} 
  \begin{flalign*}
    \hspace{4cm} L_k^{(\text{\emph{disc}})}(s) \ll |t|^{1+\frac{|\Re(u+v)| + |\Re(u-v)|}{2} + \eps}. &&
  \end{flalign*}
\end{lemma}

\noindent
Showing that the last of these bounds holds is much more involved than the other three. It uses a ``spectral large sieve'' as in e.g.\ \cite[Prop 4.1]{HH}, and is often described as essentially the Lindel\"{o}f hypothesis on average in settings where $\Re(u) = \Re(v) = 0$. 
\\
\\
\cref{L_perron_R}, \ref{L_perron_V}, \ref{L_perron_cont}, and \ref{L_perron_disc} hold for $X,T \to \infty$ with all other quantities fixed and satisfying the hypotheses of \cref{maintheorem}.
\begin{lemma}\label{L_perron_R}
  \begin{flalign*}
    \hspace{2cm} \int_{\half + \eps - iT}^{\half+\eps + iT} L_k^{(R)}(s)\,X^s\frac{ds}{s} &\ll X^{|\Re(u+v)| + \eps} + X^\eps T^\eps + X^{\half + \eps} T^{-1}. &&
  \end{flalign*}
\end{lemma}

\begin{lemma}\label{L_perron_V}
  \begin{flalign*}
    \hspace{2cm} \int_{\half + \eps - iT}^{\half+\eps + iT} L_k^{(V)}(s)\,X^s\frac{ds}{s} &\ll X^{|\Re(u-v)| + \eps} + X^{-\Re(u+v) + \eps} &&\\
    &+ X^\eps \left(T^{\half+\Re(u)} + T^{\half+\Re(v)}\right) + X^{\half + \eps}\left(T^{-\half+\Re(u)} + T^{-\half+\Re(v)}\right). &&
  \end{flalign*}
\end{lemma}

\begin{lemma}\label{L_perron_cont}
  \begin{flalign*}
    \hspace{2cm} \int_{\half + \eps - iT}^{\half+\eps + iT} L_k^{\text{\emph{(cont)}}}(s)\,X^s\frac{ds}{s} \ll \left(X^\half T^\half + X^\eps T + X^{\half + \eps}\right)T^\eps. &&
  \end{flalign*}
\end{lemma}

\begin{lemma}\label{L_perron_disc}
  \begin{flalign*}
    \hspace{2cm} \int_{\half + \eps - iT}^{\half+\eps + iT} L_k^{\text{\emph{(disc)}}}(s)\,X^s\frac{ds}{s} \ll \left(X^\half T^\half + X^\eps T + X^{\half + \eps}\right)T^{\frac{|\Re(u+v)| + |\Re(u-v)|}{2} + \eps}. &&
  \end{flalign*}
\end{lemma}

\begin{proof}[Proof of \cref{L_growth_R} and \cref{L_growth_V}]
  These follow immediately from Stirling's formula \cite[8.328.1]{GR}.
\end{proof}

\begin{proof}[Proof of \cref{L_growth_cont}]
This follows from a calculation that's very similar to the ones in \cite[Lemma 5.2]{TV}, \cite[Thm 2.1]{NPR}, and \cite[\S 5]{HLN}. Let's first consider the region of the integral where $|u|, |v| < |\omega| < |t|$. Here the exponential factor in Stirling's approximation is identically $1$. The polynomial factor is
\begin{align}\label{continuous_stirling_polynomial_part}
  2^C |t|^{\sigma-\half}|t+\omega|^{\frac{\sigma}{2}-\frac{3}{4}}|t-\omega|^{\frac{\sigma}{2}-\frac{3}{4}} \left(\prod_{\pm} |\omega + \Im(\pm u \pm v)|^{\frac{\Re(\pm u \pm v)}{2} - \frac{1}{4}}\right) \left(\prod_{\pm} |t + \Im(\pm u \pm v)|^{\half - \frac{\sigma+\Re(\pm u \pm v)}{2}}\right)
\end{align}
for some $C$ which depends only on $\sigma$ and $\Re(\pm u \pm v)$, and the products are over $\{(+u,+v), (+u,-v), (-u,+v), (-u,-v)\}$. As $|\omega|, |t| \to \infty$, the quantity above is bounded as
\begin{align*}
  \text{(\ref{continuous_stirling_polynomial_part})} \ll |t|^{\frac{3}{2} - \sigma}|t + \omega|^{\frac{\sigma}{2} - \frac{3}{4}}|t - \omega|^{\frac{\sigma}{2} - \frac{3}{4}} |\omega|^{-1}.
\end{align*}
The $L$-functions appearing in the integrand of $L_k^{(\text{cont})}(s)$ are
\begin{align*}
  \frac{\zeta(\half+i\omega+u+v)L(\half+i\omega-u+v,\chi)L(\half+i\omega+u-v,\psi)L(\half+i\omega-u-v,\chi\psi)}{L(1 + 2i\omega,\chi\psi)L(1-2i\omega,\overline{\chi\psi})}.
\end{align*}
Using the convexity bound for the $L$-functions in the numerator and the bound $L(1+2i\omega,\chi\psi) \gg |\omega|^{-\eps}$ \cite[p. 542]{DFI} for the $L$-functions in the denominator shows that the integrand of $L_k^{(\text{cont})}(s)$ is $\ll |t|^{\frac{3}{2} - \sigma}|t + \omega|^{\frac{\sigma}{2} - \frac{3}{4}}|t - \omega|^{\frac{\sigma}{2} - \frac{3}{4}}|\omega|^\eps$. Integrating this bound over the region $|\omega| < |t|$ this is $\ll |t|^{1+\eps}$.\\
\\
We note that it should be possible to improve this bound by using some spectral large sieve result instead of the convexity bound here, but in our situation the error term from the discrete spectrum is just as large, so we aren't impacted by the use of this coarse bound.\\
\\
In the region $|\omega| > |t|$, Stirling's formula has exponential decay of order $\exp[-\pi\tfrac{|\omega| - |t|}{2}]$. Using the bounds above for the size of the integrand, one sees that the integral over the region $|\omega| > |t|$ is also $\ll |t|^{1+\eps}$. Some attention is needed when $|\omega| \approx |t|$, but this is routine and discussed thoroughly in \cite[Lemma 5.2]{TV}, \cite[Thm 2.1]{NPR}, and \cite[\S 5]{HLN}. Since we aren't tracking the dependence on variables other than $|t|$, we also don't need to investigate the regions where $|\omega|$ or $|t|$ is less than $|u|$ or $|v|$.
\end{proof}

\begin{proof}[Proof of \cref{L_growth_disc}]
  The Weyl law in the present setting states that there are about $|t|$ Maass forms with eigenvalue between $|t|$ and $|t+1|$; see \cite[\S 7.7]{IK} or the appendix of \cite{HH}. Consequently, the approach we used to bound the continuous spectrum's contribution \cref{L_growth_cont}, when used for the discrete spectrum, gives a bound of $|t|^{2+\eps}$. To prove \cref{L_growth_disc}, we avoid using the convexity bound as above, and instead use a ``spectral large sieve'' like in \cite[Prop 4.1]{HH} or \cite[Lemma 3.10]{HL20}.\\
  \\
  We will need the estimate
\begin{align}\label{HL_bound}
  |r_j|^{-\eps} \ll \frac{|\rho_j(1)|^2}{\langle f_j, f_j \rangle \cosh(\pi r_j)} \ll |r_j|^\eps
\end{align}
from \cite{HL}, as well as the fact that the analytic conductor of $L(s,f_j)$ is on the order of $r_j^2$ \cite[\S 5.9]{IK}.\\
\\
We break $L_k^{(\text{disc})}(s)$ into four parts to consider separately, with each omitting any $r_j$'s from previous parts:
\begin{enumerate}
\item $|r_j| \ll 1 + |\Im(u)| + |\Im(v)|$,
\item $\big||r_j| - |t|\big| < 1$,
\item $|r_j| > |t| + 1$
\item $|r_j| < |t| - 1$.
\end{enumerate}
This approach is similar to \cite[Thm 2.1]{NPR}, \cite[\S 5]{jutila}, \cite[\S 3.4]{HL20}, and \cite[\S 5]{HLN}.\\
\\
The first three of these parts are straightforward to handle. For each of the finite number of $r_j$ satisfying $|r_j| \ll 1 + |\Im(u)| + |\Im(v)|$, Stirling's formula and \eqref{HL_bound} show that the corresponding term in $L_k^{(\text{disc})}(s)$ is bounded as $t \to \infty$. The convexity bound \cite[(5.91)]{IK} implies that the sum over terms for which $\big||r_j| - |t|\big| < 1$ will be less than the bound in \cref{L_growth_disc}. When $|r_j| > |t| + 1$, the corresponding term has exponential decay of order $\exp[-\pi\tfrac{|r_j| - |t|}{2}]$, so these will also sum to less than the bound in \ref{L_growth_disc}.\\
\\
We now bound the sum over $|r_j| < |t| - 1$. The $L$-function $L(\thalf + u + v, \overline{f_j})$ satisfies the functional equation \cite[(5.4) and p.\ 132]{IK}
\begin{align*}
  \omega(\overline{f_j})^{-1}&\frac{N^{\frac{\half + u + v}{2}}}{\pi^{\half + u + v}}\Gamma\!\left(\frac{\half + u + v + ir_j}{2}\right)\Gamma\!\left(\frac{\half + u + v - ir_j}{2}\right)L(\thalf + u + v, \overline{f_j})\\
  =\,\,&\frac{N^{\frac{\half - u - v}{2}}}{\pi^{\half - u - v}}\Gamma\!\left(\frac{\half - u - v + ir_j}{2}\right)\Gamma\!\left(\frac{\half - u - v - ir_j}{2}\right)L(\thalf - u - v, \chi\psi \times \overline{f_j}).
\end{align*}
The value of the root number $\omega(\overline{f_j})$ that appears above is given in \cite[Prop.\ 8.1]{DFI}. We will use only the fact that it's a complex number of norm $1$.\\
\\
An analogous functional equation exists for $L(\thalf+u-v,\psi\times\overline{f_j})$. In this case, one might find it more straightforward to note that in the proof of \ref{discrete_spectrum_inner_product} we made the arbitrary choice of unfolding the factor $E_{\mathbbm{1},\bar\chi}^*\!\left(z,\thalf+u\right)$ in $V$ instead of the factor $E_{\mathbbm{1},\bar\psi}^*\!\left(z,\thalf+v\right)$, and therefore each term of the sum in \ref{discrete_spectrum_contribution} is invariant under the simultaneous substitutions $u \leftrightarrow v$ and $\chi \leftrightarrow \psi$.\\
\\
It follows from the two functional equations presented above that each term of \ref{discrete_spectrum_contribution} satisfies
\begin{align}\label{symmetric_Ldisc_term}
  &\left[\int_0^\infty \int_0^1 \frac{\langle V - R, f_j\rangle}{\langle f_j, f_j\rangle}f_j(z) \,e(-kx)y^{s-1}\frac{dxdy}{y}\right]^2\\
  &= \frac{\tau(\psi)}{\sqrt{N}}\frac{\tau(\psi)}{\sqrt{N}}\omega(\overline{f_j}) \left[\frac{|\rho_j(1)|^2 a_j(k)}{\pi^{\half} k^{s-\half}\langle f_j, f_j \rangle} \frac{\Gamma\!\left(\frac{s - \half + ir_j}{2}\right)\Gamma\!\left(\frac{s - \half - ir_j}{2}\right)\Gamma(s)}{\Gamma\!\left(\frac{s+u+v}{2}\right)\Gamma\!\left(\frac{s+u-v}{2}\right)\Gamma\!\left(\frac{s-u+v}{2}\right)\Gamma\!\left(\frac{s-u-v}{2}\right)}\right]^2\prod_{\pm} \,\Gamma\!\left(\frac{\half \pm u \pm v \pm ir_j}{2}\right)\nonumber\\
  &\hspace{2cm} \cdot L(\thalf+u+v,\overline{f_j}) L(\thalf+u-v,\psi\times\overline{f_j}) L(\thalf-u+v,\chi\times\overline{f_j}) L(\thalf-u-v,\chi\psi\times\overline{f_j}),\nonumber
\end{align}
where the product is over all $8$ possible sign combinations.\\
\\
The rest of the proof of \ref{L_growth_disc} for the terms with $|r_j| < |t| - 1$ closely follows \cite[Lemma 3.10]{HL20}. Stirling's formula and \eqref{HL_bound} imply that the quantity in \eqref{symmetric_Ldisc_term} is bounded by
\begin{align*}
  &\ll |t|^{3 - 2\sigma}|t + r_j|^{\sigma - \frac{3}{2}}|t - r_j|^{\sigma - \frac{3}{2}} |r_j|^{-2}\\
  &\hspace{3cm}\cdot L(\thalf+u+v,\overline{f_j}) L(\thalf+u-v,\psi\times\overline{f_j}) L(\thalf-u+v,\chi\times\overline{f_j}) L(\thalf-u-v,\chi\psi\times\overline{f_j})
\end{align*}
for $|r_j|$ and $|t|$ large compared to $|\Im(u)| + |\Im(v)|$.\\
\\
Using Cauchy-Schwartz,
\begin{align}
\nonumber  &\sum_{|r_j| < |t|-1} \bigg|\, |t|^{3 - 2\sigma}|t + r_j|^{\sigma - \frac{3}{2}}|t - r_j|^{\sigma - \frac{3}{2}} |r_j|^{-2}\\
\nonumber  &\hspace{3cm}\cdot L(\thalf+u+v,\overline{f_j}) L(\thalf+u-v,\psi\times\overline{f_j}) L(\thalf-u+v,\chi\times\overline{f_j}) L(\thalf-u-v,\chi\psi\times\overline{f_j})\,\bigg|^\half\\
\nonumber  &\ll \left(\sum_{|r_j| < |t|-1} |t|^{3 - 2\sigma}|t + r_j|^{\sigma - \frac{3}{2}}|t - r_j|^{\sigma - \frac{3}{2}} |r_j|^{-2}\right)^\half\\
\nonumber  &\hspace{1.5cm}\left(\sum_{|r_j| < |t|-1} \left|L(\thalf+u+v,\overline{f_j})\right| \left|L(\thalf+u-v,\psi\times\overline{f_j})\right| \left|L(\thalf-u+v,\chi\times\overline{f_j})\right| \left|L(\thalf-u-v,\chi\psi\times\overline{f_j})\right|\right)^\half\\
\label{cauchyschwartz}  &\ll |t|^{\eps} \left(\sum_{|r_j| < |t|-1} \left|L(\thalf+u+v,\overline{f_j})\right|^4\right)^{\frac{1}{8}} \left(\sum_{|r_j| < |t|-1} \left|L(\thalf+u-v,\psi\times\overline{f_j})\right|^4\right)^{\frac{1}{8}}\\
\nonumber  &\hspace{1.5cm}\cdot\left(\sum_{|r_j| < |t|-1} \left|L(\thalf-u+v,\chi\times\overline{f_j})\right|^4\right)^{\frac{1}{8}} \left(\sum_{|r_j| < |t|-1} \left|L(\thalf-u-v,\chi\psi\times\overline{f_j})\right|^4\right)^{\frac{1}{8}}.
\end{align}
The approximate functional equation \cite[Thm.\ 5.3]{IK} gives an estimate for $L(s,*\times\overline{f_J})^2$, with $*$ any of $\chi$, $\psi$, $\chi\psi$, or the trivial character, as a sum of length $\ll r_j^2$. The error in this approximation is given by \cite[Prop.\ 5.4]{IK}, and is negligibly small when considered as a function only of $r_j$. Then, using Stirling's formula and the fact that $\overline{a_j(n)} = \overline{\chi\psi(n)}a_j(n)$ for $(n,N) = 1$ \cite[p.\ 132]{IK}, we find that the approximate functional equation implies that, for $0 < \sigma < 1$,
\begin{align}\label{eq:approximate_functional_equation}
  L(s,*\times\overline{f_j})^2 \ll \sum_{n=1}^{r_j^2} \left(n^{-\sigma} + |r_j|^{2-4\sigma}n^{\sigma-1}\right)b(n),
\end{align}
where $b(n)$ denotes the coefficient in the Dirichlet series of $L(s,*\times\overline{f_j})^2$, i.e.
\begin{align*}
  L(s,*\times\overline{f_j})^2 \eqqcolon \sum_{n=1}^\infty \frac{b(n)}{n^s}.
\end{align*}
Using the Hecke relations as given in \cref{sec:hecke_relations} and reasoning similarly to \cite[(3.75)]{HL20}, we have $b(n) \ll |r_j|^\eps$ for all $n \ll r_j^2$. We can now apply the spectral large sieve \cite[(3.74)]{HL20}:
\begin{align*}
  \sum_{|r_j| < |t|-1} \left|L(s,*\times\overline{f_j})^2\right|^2 &\ll \sum_{|r_j| < |t|} \left|\,\sum_{n=1}^{r_j^2} \left(n^{-\sigma} + |r_j|^{2-4\sigma}n^{\sigma-1}\right)b(n)\,\right|^2\quad\quad\text{(using \eqref{eq:approximate_functional_equation})}\\
  &\ll |t|^{2+\eps}\sum_{n=1}^{|t|^2} \left|n^{-\sigma} + |t|^{2-4\sigma}n^{\sigma-1}\right|^2\quad\quad\text{(invoking \cite[(3.74)]{HL20})}\\
  &\ll |t|^{2+\eps}\left(1 + |t|^{2-4\sigma + \eps}\right).
\end{align*}
Using this in \eqref{cauchyschwartz}, we obtain
\begin{align*}
  \eqref{cauchyschwartz} &\ll |t|^\eps \left(|t|^{8+\eps} \cdot |t|^{4|\Re(u+v)| + 4|\Re(u-v)|} \right)^{\frac{1}{8}}\\
  &\ll |t|^{1+\frac{|\Re(u+v)| + |\Re(u-v)|}{2} + \eps}.\qedhere
\end{align*}
\end{proof}

\begin{proof}[Proof of \cref{L_perron_R} and \cref{L_perron_V}]
  Set
  \begin{align*}
    I_{\text{right}} &\coloneqq \int_{\half + \eps - iT}^{\half + \eps + iT} L_k^{(R)}(s)\,X^s\frac{ds}{s} &
    I_{\text{top}} &\coloneqq \int_{\half + \eps + iT}^{\eps + iT} L_k^{(R)}(s)\,X^s\frac{ds}{s}\\
    I_{\text{left}} &\coloneqq \int_{\eps + iT}^{\eps - iT} L_k^{(R)}(s)\,X^s\frac{ds}{s} &
    I_{\text{bottom}} &\coloneqq \int_{\eps - iT}^{\half + \eps - iT} L_k^{(R)}(s)\,X^s\frac{ds}{s}.
  \end{align*}
  For $\eps < \thalf - |\Re(u)| - |\Re(v)|$, the only poles of $L_k^{(R)}(s)$ with real part between $\eps$ and $\thalf + \eps$ have real part $|\Re(u+v)|$.\\
  \\
  To bound $I_{\text{top}}$ and $I_{\text{bottom}}$, we use \cref{L_growth_R} and the Phragm\'{e}n--Lindel\"{o}f principle \cite[Thm.\ 8.2.1]{goldfeld_book}. We cannot apply the Phragm\'{e}n--Lindel\"{o}f principle to the integrand directly, as this integrand has poles in the relevant vertical strip. Let $s_1,\,\dots$ and $r_1,\,\dots$ denote the poles and associated residues of the integrand in this vertical strip. We subtract the function
  \begin{align}\label{excised_poles}
    \sum_i \frac{r_i}{s - s_i}
  \end{align}
  from the integrand. We apply the the Phragm\'{e}n--Lindel\"{o}f principle to this difference, which is now holomorphic. By inspection, the integral of \eqref{excised_poles} along the horizontal segments as well as \eqref{excised_poles}'s value at the endpoints are within the bound stated in \cref{L_perron_R}. The lemma then follows from Cauchy's residue theorem and \cref{L_growth_R} applied to $I_{\text{left}}$. The proof of \cref{L_perron_V} is similar.
\end{proof}

\begin{proof}[Proof of \cref{L_perron_cont} and \cref{L_perron_disc}]
  Just as in the proofs of \cref{L_perron_R} and \cref{L_perron_V} above, we shift the contour to the line $\Re(s) = \eps$ and use \cref{L_growth_cont} and \cref{L_growth_disc}. In both cases, the rightmost poles have real part $\thalf$. The sum of residues is bounded via the same calculations as those in the proofs of \cref{L_growth_cont} and \cref{L_growth_disc}, but with an extra factor of $|t|^\half$ coming from the $\Gamma$-function which has the pole; this $\Gamma$-function is absent when applying Stirling's formula. The horizontal segments are again bounded using the Phragm\'{e}n--Lindel\"{o}f principle, with the Weyl law ensuring that, possibly by deforming the horizontal contours slightly, the contribution from the poles that must be excised is within the bound of the lemma.
\end{proof}

\section{Partial sum asymptotics}
\noindent
\begin{definition}\label{Lstar_def}
  \begin{align*}
    &L_k^*(s) \coloneqq \sum_{n \neq 0,k}\frac{\sigma_{2u}(n,\chi)\sigma_{2v}(n-k,\psi)}{|n|^{s+u+v}}.
  \end{align*}
\end{definition}

\begin{lemma}\label{Lk_diff_holomorphic}
  The function $L_k^*(s) - L_k(s)$ is holomorphic for $\sigma > |\Re(u)| + |\Re(v)|$.
\end{lemma}
\begin{proof}
  Using the series definition of ${}_2F_1$, we see that each term in the infinite series of the difference $L_k^*(s) - L_k(s)$ is $\cO(n^{-1})$. Thus, the series converges absolutely for $\sigma > |\Re(u)| + |\Re(v)|$.
\end{proof}

\begin{corollary}\label{Lk_diff_residues}
  $L_k^*(s)$ and $L_k(s)$ have the same poles and residues in the half-plane $\sigma > |\Re(u)| + |\Re(v)|$.
\end{corollary}

\begin{lemma}\label{Lk_diff_growth}
  For any fixed $\sigma > |\Re(u)| + |\Re(v)|$ and $n > k$,
  \begin{align*}
    {}_2F_1\!\left(\frac{s+u+v}{2},\frac{s+u-v}{2} ;s\,; \frac{2k}{n} - \frac{k^2}{n^2}\right) \ll |t|^\half.
  \end{align*}
  The implied constant doesn't depend on $n$.
\end{lemma}
\begin{proof}
  Use \cite[9.111]{GR}, Stirling's formula, and the triangle inequality for the integral from $0$ to $1$.
\end{proof}

\begin{corollary}\label{Lk_diff_growth_cor}
  $L_k^*(s) - L_k(s) \ll |t|^\half$ for any fixed $\sigma > |\Re(u)| + |\Re(v)|$.
\end{corollary}

\begin{lemma}\label{rhs_step_1}
  \begin{align*}
    \int_{1 + |\Re(u)| + |\Re(v)| + \eps - iT}^{1 + |\Re(u)| + |\Re(v)| + \eps + iT} L_k^*(s) X^s\frac{ds}{s} &= \sum_{\sigma > \half + \eps} 2\pi i \,\text{\emph{Res}}_s \,L_k(s) \frac{X^s}{s}\\
    &+ \int_{\half + \eps - iT}^{\half + \eps + iT} L_k(s) X^s\frac{ds}{s}\\
    &- \left(\int_{\half + \eps - iT}^{1 + |\Re(u)| + |\Re(v)| + \eps - iT} - \int_{\half + \eps + iT}^{1 + |\Re(u)| + |\Re(v)| + \eps + iT}\right) L_k^*(s) X^s\frac{ds}{s}\\
    &+ \cO\!\left(X^{\half+\eps}T^{\half+\eps}\right).
  \end{align*}
\end{lemma}

\begin{proof}
  Shift the contour of integration from $\Re(s) = 1 + |\Re(u)| + |\Re(v)| + \eps$ to $\Re(s) = \thalf + \eps$ and apply the residue theorem. Write
  \begin{align*}
    \int_{\half + \eps - iT}^{\half + \eps + iT} L_k^*(s) X^s\frac{ds}{s} &= \int_{\half + \eps - iT}^{\half + \eps + iT} L_k(s) X^s\frac{ds}{s} + \int_{\half + \eps - iT}^{\half + \eps + iT} \big(L_k^*(s) - L_k(s)\big) X^s\frac{ds}{s}.
  \end{align*}
  Apply \ref{Lk_diff_holomorphic}, \ref{Lk_diff_residues}, and \ref{Lk_diff_growth_cor} to deduce that
  \begin{align*}
    &\sum_{\sigma > \half + \eps} 2\pi i \,\text{Res}_s \,L_k^*(s) \frac{X^s}{s} = \sum_{\sigma > \half + \eps} 2\pi i \,\text{Res}_s \,L_k(s) \frac{X^s}{s}
  \end{align*}
  and
  \begin{align*}
    &\int_{\half + \eps - iT}^{\half + \eps + iT} \big(L_k^*(s) - L_k(s)\big) X^s\frac{ds}{s} \ll X^{\half+\eps}T^{\half+\eps}.
  \end{align*}
\end{proof}

\begin{lemma}\label{Lkstar_horizontal_bound}
  \begin{align*}
    &\left(\int_{\half + \eps - iT}^{1 + |\Re(u)| + |\Re(v)| + \eps - iT} - \int_{\half + \eps + iT}^{1 + |\Re(u)| + |\Re(v)| + \eps + iT}\right) L_k^*(s) X^s\frac{ds}{s}\\
    &\hspace{6cm} \ll X^{\half + \eps} T^{\frac{|\Re(u+v)| + |\Re(u-v)|}{2} + \eps} + X^{1 + |\Re(u)| + |\Re(v)| + \eps}T^{-1}.
  \end{align*}
\end{lemma}

\begin{proof}
  \cref{L_growth_R}, \ref{L_growth_V}, \ref{L_growth_cont}, \ref{L_growth_disc}, and \ref{Lk_diff_growth_cor} imply that
  \begin{align*}
    L_k^*\!\left(\thalf + \eps + t\right) \ll |t|^{1 + \frac{|\Re(u+v)| + |\Re(u-v)|}{2} + \eps}.
  \end{align*}
  The series defining $L_k^*(s)$ is absolutely convergent on the line $\Re(s) = 1 + |\Re(u)| + |\Re(v)| + \eps$, so here the triangle inequality yields
  \begin{align*}
    L_k^*\!\left(1 + |\Re(u)| + |\Re(v)| + \eps + t\right) \ll 1.
  \end{align*}
  \cref{Lkstar_horizontal_bound} then follows from substituting the bounds above into the integrals in the lemma and applying the Phragm\'{e}n--Lindel\"{o}f principle \cite[Thm.\ 8.2.1]{goldfeld_book} after subtracting off the poles in the region (whose contributions are dwarfed by the bound in the lemma).
\end{proof}

\begin{lemma}[Perron's formula]\label{perron's_formula}
  \begin{align*}
    &\int_{1 + |\Re(u)| + |\Re(v)| + \eps - iT}^{1 + |\Re(u)| + |\Re(v)| + \eps + iT} L_k^*(s) X^s\frac{ds}{s} = 4\pi i \sum_{n=1}^X \frac{\sigma_{2u}(n,\chi)\sigma_{2v}(n-k,\psi)}{n^{u+v}} + \cO\!\left(X^{1 + |\Re(u)| + |\Re(v)| + \eps}T^{-1}\right).
  \end{align*}
\end{lemma}
\begin{proof}
  \cite[\S 5.1]{montgomery_vaughan}
\end{proof}

\begin{proof}[Proof of \cref{maintheorem}]
  Our proof is similar to \cite[Prop.\ 2.3]{NPR}. Using \cref{perron's_formula}, we find that the integral
  \begin{align}\label{mainintegral}
    \frac{1}{4\pi i}\int_{1 + |\Re(u)| + |\Re(v)| + \eps - iT}^{1 + |\Re(u)| + |\Re(v)| + \eps + iT} L_k^*(s) X^s\frac{ds}{s}
  \end{align}
  is equal to the left hand side of \cref{maintheorem} with an error of $\cO\!\left(X^{1 + |\Re(u)| + |\Re(v)| + \eps}T^{-1}\right)$. The remainder of this proof shows that \eqref{mainintegral} is also equal to the right hand side of \cref{maintheorem}.\\
  \\
  We first apply \cref{rhs_step_1}. By inspection of \cref{L_parts_def}, we see that the term
  \begin{align*}
    \sum_{\sigma > \half + \eps} 2\pi i \,\text{Res}_s \,L_k(s) \frac{X^s}{s}
  \end{align*}
  on the right hand side of \ref{rhs_step_1} yields the main term of the right hand side of \cref{maintheorem}. The term
  \begin{align*}
    \left(\int_{\half + \eps - iT}^{1 + |\Re(u)| + |\Re(v)| + \eps - iT} - \int_{\half + \eps + iT}^{1 + |\Re(u)| + |\Re(v)| + \eps + iT}\right) L_k^*(s) X^s\frac{ds}{s}
  \end{align*}
  on the right hand side of \ref{rhs_step_1} is bounded in \cref{Lkstar_horizontal_bound}.\\
  \\
  We now estimate
  \begin{align}\label{Lk_half_int}
    \int_{\half + \eps - iT}^{\half + \eps + iT} L_k(s) X^s\frac{ds}{s}.
  \end{align}
  Write
  \begin{align*}
    L_k(s) = L_k^{(\text{cont})}(s) + L_k^{(\text{disc})}(s) + L_k^{(R)}(s) + L_k^{(V)}(s).
  \end{align*}
  The quantities on the right are defined in \ref{L_parts_def}. Substitute the above into \eqref{Lk_half_int} and apply \ref{L_perron_R}, \ref{L_perron_V}, \ref{L_perron_cont}, and \ref{L_perron_disc}.\\
  \\
  In total, we have that the left hand side of \cref{maintheorem} comes from Perron's formula \ref{perron's_formula}, and the main term of the right hand side comes from the term
  \begin{align*}
    \sum_{\sigma > \half + \eps} 2\pi i \,\text{Res}_s \,L_k(s) \frac{X^s}{s}
  \end{align*}
  appearing on the right hand side of \ref{rhs_step_1}. The error terms which appear are
  \begin{alignat*}{2}
    &\cO\!\left(X^{1 + |\Re(u)| + |\Re(v)| + \eps}T^{-1}\right) && \text{from \ref{perron's_formula}}\\
    &\cO\!\left(X^{\half+\eps}T^{\half+\eps}\right) && \text{from \ref{rhs_step_1}}\\
    &\cO\!\left(X^{\half + \eps} T^{\frac{|\Re(u+v)| + |\Re(u-v)|}{2} + \eps} + X^{1 + |\Re(u)| + |\Re(v)| + \eps}T^{-1}\right) && \text{from \ref{Lkstar_horizontal_bound}}\\
    &\cO\!\left(X^{|\Re(u+v)| + \eps} + X^\eps T^\eps + X^{\half + \eps} T^{-1}\right) && \text{from \ref{L_perron_R}}\\
    &\cO\!\left(\left(X^{|\Re(u-v)|} + X^{-\Re(u+v)} + T^{\half+\Re(u)} + T^{\half+\Re(v)} + X^\half T^{-\half+\Re(u)} + X^\half T^{-\half+\Re(v)}\right)X^\eps\right) \hspace{1cm}&& \text{from \ref{L_perron_V}}\\
    &\cO\!\left(\left(X^\half T^\half + X^\eps T + X^{\half + \eps}\right)T^\eps\right) && \text{from \ref{L_perron_cont}}\\
    &\cO\!\left(\left(X^\half T^\half + X^\eps T + X^{\half + \eps}\right)T^{\frac{|\Re(u+v)| + |\Re(u-v)|}{2} + \eps}\right) && \text{from \ref{L_perron_disc}}.
  \end{alignat*}
  Making the substitution $T = X^\alpha$, with
  \begin{align*}
    \alpha \coloneqq \frac{1 + 2|\Re(u)| + 2|\Re(v)|}{3 + |\Re(u+v)| + |\Re(u-v)|},
  \end{align*}
  then yields \cref{maintheorem}.
\end{proof}

\bibliographystyle{alpha}
\bibliography{additivedivisorbib}{}

\begin{thebibliography}{HKLDW18}

\bibitem[BD16]{BD}
Roelof Bruggeman and Nikolaos Diamantis.
\newblock Fourier coefficients of {E}isenstein series formed with modular
  symbols and their spectral decomposition.
\newblock {\em J. Number Theory}, 167:317--335, 2016.

\bibitem[BH08]{BH}
Valentin Blomer and Gergely Harcos.
\newblock The spectral decomposition of shifted convolution sums.
\newblock {\em Duke Math. J.}, 144(2):321--339, 2008.

\bibitem[BM01]{BM}
R.~W. Bruggeman and Y.~Motohashi.
\newblock Fourth power moment of {D}edekind zeta-functions of real quadratic
  number fields with class number one.
\newblock {\em Funct. Approx. Comment. Math.}, 29:41--79, 2001.

\bibitem[BMS19]{BMS}
Yu.~A. Brychkov, O.~I. Marichev, and N.~V. Savischenko.
\newblock {\em Handbook of {M}ellin transforms}.
\newblock Advances in Applied Mathematics. CRC Press, Boca Raton, FL, 2019.

\bibitem[DFI02]{DFI}
W.~Duke, J.~B. Friedlander, and H.~Iwaniec.
\newblock The subconvexity problem for {A}rtin {$L$}-functions.
\newblock {\em Invent. Math.}, 149(3):489--577, 2002.

\bibitem[Gol81]{goldfeld}
Dorian Goldfeld.
\newblock On convolutions of nonholomorphic {E}isenstein series.
\newblock {\em Adv. in Math.}, 39(3):240--256, 1981.

\bibitem[Gol06]{goldfeld_book}
Dorian Goldfeld.
\newblock {\em Automorphic forms and {$L$}-functions for the group {${\rm
  GL}(n,\bold R)$}}, volume~99 of {\em Cambridge Studies in Advanced
  Mathematics}.
\newblock Cambridge University Press, Cambridge, 2006.
\newblock With an appendix by Kevin A. Broughan.

\bibitem[Goo83]{good}
Anton Good.
\newblock On various means involving the {F}ourier coefficients of cusp forms.
\newblock {\em Math. Z.}, 183(1):95--129, 1983.

\bibitem[GR07]{GR}
I.~S. Gradshteyn and I.~M. Ryzhik.
\newblock {\em Table of integrals, series, and products}.
\newblock Elsevier/Academic Press, Amsterdam, seventh edition, 2007.
\newblock Translated from the Russian, Translation edited and with a preface by
  Alan Jeffrey and Daniel Zwillinger, With one CD-ROM (Windows, Macintosh and
  UNIX).

\bibitem[Har03]{harcos}
Gergely Harcos.
\newblock An additive problem in the {F}ourier coefficients of cusp forms.
\newblock {\em Math. Ann.}, 326(2):347--365, 2003.

\bibitem[HH16]{HH}
Jeff Hoffstein and Thomas~A. Hulse.
\newblock Multiple {D}irichlet series and shifted convolutions.
\newblock {\em J. Number Theory}, 161:457--533, 2016.
\newblock With an appendix by Andre Reznikov.

\bibitem[HKLDW18]{HKLDW}
Thomas~A. Hulse, Chan~Ieong Kuan, David Lowry-Duda, and Alexander Walker.
\newblock Second moments in the generalized {G}auss circle problem.
\newblock {\em Forum Math. Sigma}, 6:Paper No. e24, 49, 2018.

\bibitem[HL94]{HL}
Jeffrey Hoffstein and Paul Lockhart.
\newblock Coefficients of {M}aass forms and the {S}iegel zero.
\newblock {\em Ann. of Math. (2)}, 140(1):161--181, 1994.
\newblock With an appendix by Dorian Goldfeld, Hoffstein and Daniel Lieman.

\bibitem[HL20]{HL20}
Jeff Hoffstein and Min Lee.
\newblock Second moments of rankin-selberg convolutions and shifted dirichlet
  series, 2020.
\newblock arXiv 2008.12040.

\bibitem[HLN21]{HLN}
Jeff Hoffstein, Min Lee, and Maria Nastasescu.
\newblock First moments of {R}ankin-{S}elberg convolutions of automorphic forms
  on {$\rm GL(2)$}.
\newblock {\em Res. Number Theory}, 7(4):Paper No. 60, 44, 2021.

\bibitem[IK04]{IK}
Henryk Iwaniec and Emmanuel Kowalski.
\newblock {\em Analytic number theory}, volume~53 of {\em American Mathematical
  Society Colloquium Publications}.
\newblock American Mathematical Society, Providence, RI, 2004.

\bibitem[Iwa97]{iwaniec_topics}
Henryk Iwaniec.
\newblock {\em Topics in classical automorphic forms}, volume~17 of {\em
  Graduate Studies in Mathematics}.
\newblock American Mathematical Society, Providence, RI, 1997.

\bibitem[Iwa02]{iwaniec}
Henryk Iwaniec.
\newblock {\em Spectral methods of automorphic forms}, volume~53 of {\em
  Graduate Studies in Mathematics}.
\newblock American Mathematical Society, Providence, RI; Revista Matem\'{a}tica
  Iberoamericana, Madrid, second edition, 2002.

\bibitem[JM95]{JM}
Matti Jutila and Yoichi Motohashi.
\newblock Mean value estimates for exponential sums and {$L$}-functions: a
  spectral-theoretic approach.
\newblock {\em J. Reine Angew. Math.}, 459:61--87, 1995.

\bibitem[Jut96]{jutila}
Matti Jutila.
\newblock The additive divisor problem and its analogs for {F}ourier
  coefficients of cusp forms. {I}.
\newblock {\em Math. Z.}, 223(3):435--461, 1996.

\bibitem[Mic04]{michel:subconvexity}
P.~Michel.
\newblock The subconvexity problem for {R}ankin-{S}elberg {$L$}-functions and
  equidistribution of {H}eegner points.
\newblock {\em Ann. of Math. (2)}, 160(1):185--236, 2004.

\bibitem[Mic07]{michel}
Philippe Michel.
\newblock Analytic number theory and families of automorphic {$L$}-functions.
\newblock In {\em Automorphic forms and applications}, volume~12 of {\em
  IAS/Park City Math. Ser.}, pages 181--295. Amer. Math. Soc., Providence, RI,
  2007.

\bibitem[Mot94]{motohashi}
Y{\={o}}ichi Motohashi.
\newblock The binary additive divisor problem.
\newblock {\em Ann. Sci. \'{E}cole Norm. Sup. (4)}, 27(5):529--572, 1994.

\bibitem[M{\"{u}}l89]{mueller}
Wolfgang M{\"{u}}ller.
\newblock On the asymptotic behaviour of the ideal counting function in
  quadratic number fields.
\newblock {\em Monatsh. Math.}, 108(4):301--323, 1989.

\bibitem[MV07]{montgomery_vaughan}
Hugh~L. Montgomery and Robert~C. Vaughan.
\newblock {\em Multiplicative number theory. {I}. {C}lassical theory},
  volume~97 of {\em Cambridge Studies in Advanced Mathematics}.
\newblock Cambridge University Press, Cambridge, 2007.

\bibitem[MV10]{MV}
Philippe Michel and Akshay Venkatesh.
\newblock The subconvexity problem for {${\rm GL}_2$}.
\newblock {\em Publ. Math. Inst. Hautes \'{E}tudes Sci.}, (111):171--271, 2010.

\bibitem[Nel19]{nelson}
Paul~D. Nelson.
\newblock Eisenstein series and the cubic moment for $\text{PGL}_2$, 2019.
\newblock arXiv 1911.06310.

\bibitem[NPR22]{NPR}
Asbj\o rn~C. Nordentoft, Yiannis~N. Petridis, and Morten~S. Risager.
\newblock Bounds on shifted convolution sums for {H}ecke eigenforms.
\newblock {\em Res. Number Theory}, 8(2):Paper No. 26, 20, 2022.

\bibitem[PR04]{PR}
Y.~N. Petridis and M.~S. Risager.
\newblock Modular symbols have a normal distribution.
\newblock {\em Geom. Funct. Anal.}, 14(5):1013--1043, 2004.

\bibitem[Tem11]{templier}
Nicolas Templier.
\newblock Heegner points and {E}isenstein series.
\newblock {\em Forum Math.}, 23(6):1135--1158, 2011.

\bibitem[VT84]{TV}
A.~I. Vinogradov and L.~A. Takhtadzhyan.
\newblock The zeta function of the additive divisor problem and spectral
  expansion of the automorphic {L}aplacian.
\newblock volume 134, pages 84--116. 1984.
\newblock Automorphic functions and number theory, II.

\bibitem[Wil23]{wilson}
B.~M. Wilson.
\newblock Proofs of {S}ome {F}ormulae {E}nunciated by {R}amanujan.
\newblock {\em Proc. London Math. Soc. (2)}, 21:235--255, 1923.

\bibitem[Wu19]{wu}
Han Wu.
\newblock Deducing {S}elberg trace formula via {R}ankin-{S}elberg method for
  {$\rm{GL}_2$}.
\newblock {\em Trans. Amer. Math. Soc.}, 372(12):8507--8551, 2019.

\bibitem[You11]{young:4thmoment}
Matthew~P. Young.
\newblock The fourth moment of {D}irichlet {$L$}-functions.
\newblock {\em Ann. of Math. (2)}, 173(1):1--50, 2011.

\bibitem[You16]{young:QUE}
Matthew~P. Young.
\newblock The quantum unique ergodicity conjecture for thin sets.
\newblock {\em Adv. Math.}, 286:958--1016, 2016.

\bibitem[You19]{young}
Matthew~P. Young.
\newblock Explicit calculations with {E}isenstein series.
\newblock {\em J. Number Theory}, 199:1--48, 2019.

\bibitem[Zag81]{zagier}
Don Zagier.
\newblock The {R}ankin-{S}elberg method for automorphic functions which are not
  of rapid decay.
\newblock {\em J. Fac. Sci. Univ. Tokyo Sect. IA Math.}, 28(3):415--437 (1982),
  1981.

\bibitem[Zag85]{zagier2}
Don Zagier.
\newblock Modular parametrizations of elliptic curves.
\newblock {\em Canad. Math. Bull.}, 28(3):372--384, 1985.

\end{thebibliography}

\end{document}